\documentclass[a4paper,11pt]{article}
\usepackage[utf8]{inputenc}
\usepackage[T1]{fontenc}
\usepackage[english]{babel}
\usepackage{amsmath,amsthm,amssymb}
\usepackage{url}
\usepackage{eucal}
\usepackage{fancyhdr}
\usepackage{color}
\usepackage{tikz}
\usepackage{graphicx}
\usepackage[font=small]{caption}
\usepackage[caption=false]{subfig}

\newtheorem{theorem}{Theorem}
\newtheorem{lemma}{Lemma}
\newtheorem{corollary}{Corollary}
\newtheorem{remark}{Remark}
\numberwithin{equation}{section}

\newcommand{\card}{\operatorname{card}}
\newcommand{\supp}{\operatorname{supp}}
\newcommand{\G}{\mathcal{G}}
\newcommand{\N}{\mathcal{N}}
\newcommand{\Cor}{\mathcal{C}}
\newcommand{\cR}{\mathcal{R}}
\newcommand{\cS}{\mathcal{S}}
\newcommand{\CIh}{C_{I_h}}
\newcommand{\CIH}{C_{I_H}}
\newcommand{\Col}{C_{\mathrm{ol}}}
\newcommand{\Colm}{C_{\mathrm{ol},m}}
\newcommand{\nei}{\mathsf{N}}

\begin{document}
% header fuer unsere persoenliche versionsuebersicht
\pagestyle{fancy}
\fancyhead{}
\setlength{\headheight}{14pt}
\renewcommand{\headrulewidth}{0pt}
\fancyhead[c]{\small \it Multiscale Petrov-Galerkin FEM}
%\fancyhead[R]{\it\tiny\today}

\title{Stable Multiscale Petrov-Galerkin Finite Element Method
       for High Frequency Acoustic Scattering}

\author{%
       D.~Gallistl\thanks{Institut f\"ur Numerische Simulation,
         Universit\"at Bonn,        
         Wegelerstra{\ss}e 6, D-53115 Bonn, Germany,
         \texttt{\{gallistl,peterseim\}@ins.uni-bonn.de}}
        \and 
       D.~Peterseim\footnotemark[1]
       }
\date{}

\maketitle

\begin{abstract}
We present and analyze a pollution-free Petrov-Galerkin multiscale finite
element method 
for the Helmholtz problem with large wave number $\kappa$
as a variant of [Peterseim, ArXiv:1411.1944, 2014].
We use standard continuous $Q_1$ finite elements at a coarse
discretization
scale $H$ as trial functions, whereas the test functions are computed as 
the solutions of local problems at a finer scale $h$.
The diameter of the support of the test functions behaves like $mH$ for 
some oversampling parameter $m$.
Provided 
$m$ is of the order of $\log(\kappa)$
and $h$ is sufficiently small,
the resulting method is stable and quasi-optimal in the regime where
$H$ is proportional to $\kappa^{-1}$.
In homogeneous (or more general periodic) media, the fine scale test 
functions depend only on local mesh-configurations.
Therefore, the seemingly high cost for the computation of the
test functions can be drastically reduced on structured meshes.
We present numerical experiments in two and three space dimensions.
\end{abstract}

{\small
\noindent
\textbf{Keywords} multiscale method, pollution effect, 
 wave propagation, Helmholtz problem, finite element method

\noindent
\textbf{AMS subject classification}
35J05,  %Laplacian operator, reduced wave equation (Helmholtz equation), Poisson equation
65N12,  %Stability and convergence of numerical methods
65N15,  %Error bounds
65N30   %Finite elements, Rayleigh-Ritz and Galerkin methods, finite methods
}

\section{Introduction}

Standard finite element methods (FEMs) for acoustic wave propagation
are well known to exhibit the so-called \emph{pollution effect}
\cite{Babuska:2000:PEF:354138.354150},
which means that the stability and convergence of the scheme
require a much smaller mesh-size than needed for a meaningful 
approximation of the wave by finite element functions.
For an highly oscillatory wave at wave number $\kappa$, the typical
requirement for a reasonable representation reads 
$\kappa H\lesssim 1$ for the mesh-size $H$,
that is some fixed number of elements per wave-length. 
The standard Galerkin FEM typically requires at least
$\kappa^\alpha H\lesssim 1$ where $\alpha>1$ depends on
the method and the stability and regularity properties of the continuous problem.
There have been various attempts to reduce or avoid the pollution 
effect, e.g.,
discontinuous Galerkin methods
\cite{MR2219901,MR2551150,MR2813347,perugia},
high-order finite elements \cite{MS10,MS11},
discontinuous Petrov-Galerkin methods
\cite{Zitelli20112406,dpg},
or the continuous interior penalty method
\cite{Wu2014CIPFEM}
among many others. A good historical overview is provided in
\cite{Zitelli20112406}.

\begin{figure}
\begin{center}
\includegraphics[width=.6\textwidth]{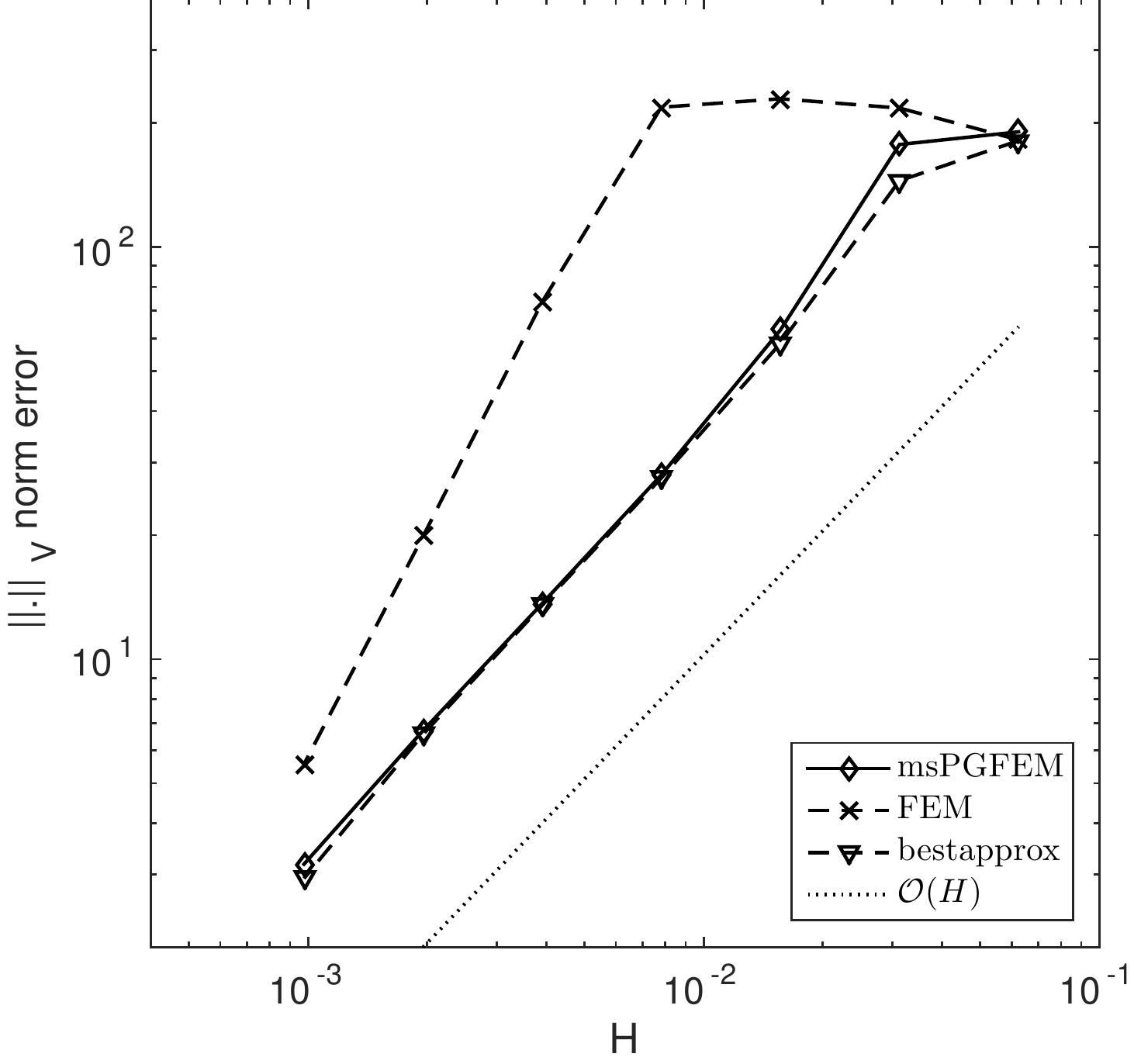}
\caption{ Convergence history of the multiscale FEM (msPGFEM), the standard
          $Q_1$ FEM (FEM) and the best-approximation (bestapprox) in the
          finite element space for a two-dimensional plane wave
          with wave number $\kappa=2^7$ (see also Section~\ref{s:num}).
         \label{f:convhistintro}}
\end{center}
\end{figure}

The work \cite{Peterseim2014} suggested a multiscale 
Petrov-Galerkin method for the Helmholtz equation
where standard finite element 
trial and test functions are modified by a local subscale correction
in the spirit of numerical homogenization \cite{MP14}.
In the numerical experiments of \cite{Peterseim2014}, a variant of that
method appeared attractive where only the test functions are 
modified while standard finite element functions are used as 
trial functions.
In this paper, we analyze that method and reformulate it
as a stabilized $Q_1$ method in the spirit of the variational multiscale
method \cite{Hughes:1995,MR1660141,MR2300286,Malqvist:2011,Peterseim2015}.
The method employs standard $Q_1$ finite element trial functions
on a grid $\G_H$ with mesh-size $H$. The test functions are the 
solutions of local problems with respect to a grid $\G_h$
at a finer scale $h$ which is chosen fine enough to allow for 
stability of the standard Galerkin FEM over $\G_h$.
The diameter of the support of the test functions is proportional to
$mH$ for the oversampling parameter $m$.
Under the condition that $m$ is
logarithmically coupled with the wave number $\kappa$
through  $m\approx \log(\kappa)$,
we prove that the method is pollution-free, i.e., the resolution
condition
$\kappa H\lesssim 1$ is sufficient for stability and quasi-optimality
under fairly general assumptions on the stability of the continuous
problem.
The performance of the method is illustrated in the convergence
history of Figure~\ref{f:convhistintro}. More detailed descriptions
on the numerical experiments will be given in Section~\ref{s:num}.
As the test functions only depend on local mesh-configurations,
on structured meshes the number of test functions to be actually computed
is much smaller then the overall number of trial and test functions 
on the coarse scale.
In many cases,
the computational cost is then dominated by the coarse solve
and the overhead compared with a standard FEM on the same coarse
mesh remains proportional to $m^d\approx\log(\kappa)^d$.
Even if no structure of the mesh can be exploited to reduce the
number of patch problems, the method may still be attractive if
the problem has to be solved many times with different data
(same $\kappa$ though) in the context of inverse problems or
parameter identification problems.

\medskip
The paper is structured as follows. Section~\ref{s:Helmholtz}
states the Helmholtz
problem and recalls some important results.
The definition of the new Petrov-Galerkin method follows in 
Section~\ref{s:method}.
Stability and error analysis are carried out in Section~\ref{s:erroranalysis}.
Section~\ref{s:num} is devoted to numerical experiments.

\bigskip
Standard notation on complex-valued Lebesgue and Sobolev spaces
applies throughout this paper. 
The bar indicates
complex conjugation and $i$ is the imaginary unit.
The $L^2$ inner product is denoted
by 
$(v,w)_{L^2(\Omega)}:=\int_\Omega v\bar w\,dx$. 
The Sobolev space of complex-valued $L^p$ functions over a
domain $\omega$ whose 
generalized derivatives up to order $k$ belong to $L^p$ is
denoted by $W^{k,p}(\omega;\mathbb C)$.
The notation $A\lesssim B$ abbreviates $A\leq C B$ for some
constant $C$ that is independent of the mesh-size,
the wave number $\kappa$, and all further parameters in the method like
the oversampling parameter $m$ or the fine-scale mesh-size $h$;
$A\approx B$ abbreviates $A\lesssim B\lesssim A$.

\section{The Helmholtz Problem}\label{s:Helmholtz}

Let $\Omega\subseteq\mathbb R^d$, for $d\in\{1,2,3\}$, be an open
bounded domain with polyhedral Lipschitz boundary which is decomposed
into disjoint parts $\partial\Omega =\Gamma_D \cup\Gamma_R$ with
$\Gamma_D$ closed.
The classical Helmholtz equation then reads
\begin{equation}\label{e:HelmholtzClassic}
\begin{aligned}
 -\Delta u - \kappa^2 u &&=&& & f &&\text{in } \Omega, \\
 u && = && & u_D && \text{on } \Gamma_D, \\
 i\kappa u -\nabla u \cdot \nu && = && & g && \text{on } \Gamma_R
\end{aligned}
\end{equation}
for the outer unit normal $\nu$ of $\Omega$
and the real parameter $\kappa>0$.
For the sake of a simple exposition we assume $u_D=0$.
Either of the parts $\Gamma_D$ or $\Gamma_R$ is allowed to be
the empty set.
In scattering problems, the Dirichlet boundary $\Gamma_D$ typically
refers to the boundary of a bounded sound-soft object
whereas the Robin boundary $\Gamma_R$ arises from artificially
truncating the full space $\mathbb R^d$ to the bounded domain 
$\Omega$ \cite{Ihlenburg1998}.
The variational formulation of \eqref{e:HelmholtzClassic} employs
the space
\begin{equation*}
V
:=W^{1,2}_D(\Omega;\mathbb C)
:=\{ v\in W^{1,2}(\Omega;\mathbb C) \,:\, v|_{\Gamma_D} = 0\} .
\end{equation*}
For any subset $\omega\subseteq\Omega$ we define the norm
\begin{equation*}
\| v \|_{V,\omega} :=
\sqrt{ \kappa^2 \| v\|_{L^2(\omega)}^2 + \|\nabla v\|_{L^2(\omega)}^2 }
\quad\text{for any } v\in V
\end{equation*}
and denote $\| v \|_V:=\| v \|_{V,\Omega} $.
Define on $V$ the following sesquilinear form
\begin{equation*}
a(v,w):= (\nabla v,\nabla w)_{L^2(\Omega)}
        - \kappa^2(v,w)_{L^2(\Omega)}
        - i\kappa(v,w)_{L^2(\Gamma_R)}.
\end{equation*}
Although the results of this paper hold for a rather general right-hand
side in the dual of $V$, we focus on
data $f\in L^2(\Omega;\mathbb C)$ and 
$g\in L^2(\Gamma_R;\mathbb C)$ for the ease of presentation.
The weak form of the Helmholtz problem then seeks $u\in V$ such that
\begin{equation}\label{e:HelmholtzWeak}
a(u,v) =
(f,v)_{L^2(\Omega)} + (g,v)_{L^2(\Gamma_R)}
\quad\text{for all } v\in V.
\end{equation}
We assume that the problem is polynomially well-posed 
\cite{MelenkEsterhazy} in the sense that there
exists some constant $\gamma(\kappa,\Omega)$ which depends
polynomially on $\kappa$ such that
\begin{equation}\label{e:helmholtzstability}
\gamma(\kappa,\Omega)^{-1}
  \leq
  \inf_{v\in V\setminus\{0\}} \sup_{w\in V\setminus\{0\}} 
   \frac{\Re a(v,w)}{\|v\|_V\|w\|_V} .
\end{equation}
For instance, in the particular case of pure impedance boundary conditions
$\partial\Omega=\Gamma_R$, it was proved in 
\cite{melenk_phd,feng} 
by employing a technique of \cite{makridakis} that 
$\gamma(\kappa,\Omega)\lesssim \kappa$.
Further setups allowing for polynomially well-posedness are described in
\cite{hetmaniuk,MelenkEsterhazy,hiptmair}.
In particular, the case of a medium described by a convex domain
(with Robin boundary conditions on the outer part of the boundary)
and a star-shaped scatterer (with Dirichlet boundary conditions)
allows for polynomial well-posedness \cite{hetmaniuk}.
Another admissible setting is described in \cite{MelenkEsterhazy}
where $\Omega$ is a bounded Lipschitz domain with pure
Robin boundary.
For general configurations, however, the dependence of 
the stability constant $\gamma(\kappa,\Omega)$
from \eqref{e:helmholtzstability} is an open question.
Throughout this paper we assume that \eqref{e:helmholtzstability}
is satisfied.
The case of a possible exponential dependence \cite{betcke} is
excluded here.

\section{The Method}\label{s:method}
This section introduces the notation on finite element spaces and
meshes and defines the multiscale Petrov-Galerkin method 
(msPGFEM) for the Helmholtz problem.

\subsection{Meshes and Data Structures}
Let $\G_H$ be a regular partition of $\Omega$ into 
intervals, parallelograms, parallelepipeds for
$d=1,2,3$, respectively, such that $\cup\G_H =\overline\Omega$
and any two distinct $T,T'\in\G_H$ are either disjoint or share
exactly one lower-dimensional hyper-face
(that is a vertex or an edge for $d\in\{2,3\}$ or a face
for $d=3$).
We impose shape-regularity in the sense that the aspect ratio of the
elements in $\G_H$ is uniformly bounded.
Since we are considering quadrilaterals (resp. hexahedra) with
parallel faces, this guarantees the non-degeneracy of the elements in $\G_H$.
We consider this type of partitions for the sake of a simple presentation
and to exploit the structure to increase the computational efficiency.
The theory of this paper carries over to simplicial triangulations
or to more general quadrilateral or hexahedral
partitions satisfying suitable non-degeneracy conditions
or even to meshless methods based on proper partitions of unity
\cite{HMP14}.

Given any subdomain $S\subseteq\overline\Omega$, define its neighbourhood
via
\begin{equation*}
\nei(S):=\operatorname{int}
          \Big(\cup\{T\in\G_H\,:\,T\cap\overline S\neq\emptyset  \}\Big).
\end{equation*}
Furthermore, we introduce for any $m\geq 2$ the patches
\begin{equation*}
\nei^1(S):=\nei(S)
\qquad\text{and}\qquad
\nei^m(S):=\nei(\nei^{m-1}(S)) .
\end{equation*}
The shape-regularity implies that there is a uniform bound 
$\Colm=\Colm(d)$
on the number of elements in the $m$th-order patch,
\begin{equation*}
\max_{T\in\G_H}\card\{ K\in\G_H\,:\, K\subseteq \overline{\nei^m(T)}\}
\leq \Colm.
\end{equation*}
We abbreviate $\Col:=C_{\mathrm{ol},1}$.
Throughout this paper, we assume that the coarse-scale mesh $\G_H$
is quasi-uniform. This implies that $\Colm$ depends polynomially
on $m$.
The global mesh-size reads 
$H:=\max\{\operatorname{diam}(T):T\in\G_H\}$.
Let $Q_p(\G_H)$ denote the space of piecewise polynomials of partial
degree $\leq p$.
The space of globally continuous piecewise first-order polynomials reads
\begin{equation*}
\cS^1(\G_H):= C^0(\Omega)\cap Q_1(\G_H).
\end{equation*}
The standard $Q_1$ finite element space reads
\begin{equation*}
V_H:=\cS^1(\G_H) \cap V.
\end{equation*}
The set of free vertices (the degrees of freedom) is denoted by
$$
  \N_H:=\{z\in\overline\Omega\,:\, 
           z\text{ is a vertex of }\G_H\text{ and }z\notin\Gamma_D\}.
$$
Let $I_H:V\to V_H$ be a 
surjective
quasi-interpolation operator that
acts as a stable quasi-local projection in the sense that
$I_H\circ I_H = I_H$ and that
for any $T\in\G_H$ and all $v\in V$ there holds
\begin{equation}\label{e:IHapproxstab}
H^{-1}\|v-I_H v\|_{L^2(T)} + \|\nabla I_H v \|_{L^2(T)}
\leq \CIH \|\nabla v\|_{L^2(\nei(T))} .
\end{equation}
Under the mesh condition that $\kappa H \lesssim 1$ is bounded by a generic constant,
this implies  stability in the $\|\cdot\|_V$ norm
\begin{equation}\label{e:IHapproxstabV}
\|I_H v\|_V \leq C_{I_H,V} \|v\|_V
\quad\text{for all } v\in V,
\end{equation}
with a $\kappa$-independent constant $C_{I_H,V}$.
One possible choice (which we use in our implementation of the method)
is to define $I_H:=E_H\circ\Pi_H$, where
$\Pi_H$ is the piecewise $L^2$ projection onto $Q_1(\G_H)$
and $E_H$ is the averaging operator that maps $Q_1(\G_H)$ to $V_H$ by
assigning to each free vertex the arithmetic mean of the corresponding
function values of the neighbouring cells, that is, for any $v\in Q_1(\G_H)$
and any free vertex $z\in\N_H$,
\begin{equation*}
(E_H(v))(z) =
           \sum_{\substack{T\in\G_H\\\text{with }z\in T}}v|_T (z) 
           \bigg/
           \card\{K\in\G_H\,:\,z\in K\}.
\end{equation*}
Note that $E_H(v)|_{\Gamma_D} = 0$ by construction.
For this choice,
the proof of \eqref{e:IHapproxstab} follows from combining the
well-established approximation and stability properties of 
$\Pi_H$ and $E_H$, see, e.g., \cite{ern}.

\subsection{Definition of the Method}\label{ss:defMethod}

The method is determined by three parameters,
namely the coarse-scale mesh-size $H$, and the stabilization parameters 
$h$ (the fine-scale mesh-size) and $m$ (the oversampling parameter)
which are explained in the following.
We assign to any $T\in\G_H$ its $m$-th order patch
$\Omega_T:=\nei^m(T)$ (for a positive integer $m$)
and define for any $v,w\in V$ the localized sesquilinear forms
\begin{equation*}
a_{\Omega_T}(v,w):= (\nabla v,\nabla w)_{L^2(\Omega_T)}
        - (\kappa^2 v,w)_{L^2(\Omega_T)}
        - i(\kappa v,w)_{L^2(\Gamma_R\cap\partial\Omega_T)}
\end{equation*}
and
\begin{equation*}
a_T(v,w):= (\nabla v,\nabla w)_{L^2(T)}
        - (\kappa^2 v,w)_{L^2(T)}
        - i(\kappa v,w)_{L^2(\Gamma_R\cap\partial T)}.
\end{equation*}
Let $\G_h$ be a global uniform refinement of the mesh $\G_H$ over
$\Omega$ and define
\begin{equation*}
V_h(\Omega_T) 
 := \{ v\in Q_1(\G_h) \cap V\,: v=0\text{ outside }\Omega_T\} .
\end{equation*}
Define the null space
\begin{equation*}
W_h(\Omega_T) := \{ v_h\in V_h(\Omega_T) \,:\, I_H(v_h) = 0\}
\end{equation*}
of the quasi-interpolation operator $I_H$ defined in the previous section.
Given any nodal basis function $\Lambda_z\in V_H$,
let 
$\lambda_{z,T}\in W_h(\Omega_T)$
solve the subscale corrector problem
\begin{equation}\label{e:lambdacorrectorproblem}
a_{\Omega_T}(w,\lambda_{z,T}) = a_T(w,\Lambda_z)
\quad\text{for all } w\in W_h(\Omega_T).
\end{equation}
The well-posedness of \eqref{e:lambdacorrectorproblem} will
be proved in Section~\ref{s:erroranalysis}.
Let $\lambda_z:=\sum_{T\in\G_H} \lambda_{z,T}$
and define the test function
\begin{equation*}
\widetilde\Lambda_z := \Lambda_z -  \lambda_z.
\end{equation*}
The space of test functions then reads
\begin{equation*}
\widetilde V_H := \operatorname{span}\{\widetilde\Lambda_z\,:\,z\in\N_H\} .
\end{equation*}
We emphasize that the dimension
$\dim V_H = \dim \widetilde V_H$
is independent of the parameters $m$ and $h$.
Figures~\ref{f:corrector1D}--\ref{f:corrector2D} display 
typical examples for the test functions $\widetilde\Lambda_z$
and correctors.
The multiscale Petrov-Galerkin FEM seeks $u_H\in V_H$ such that
\begin{equation}\label{e:discreteproblem}
a(u_H,\tilde v_H) = 
(f,\tilde v_H)_{L^2(\Omega)}
  + (g,\tilde v_H)_{L^2(\Gamma_R)}
\quad\text{for all } \tilde v_H\in \widetilde V_H.
\end{equation}

The error analysis and the numerical experiments will show that
the choice $H\lesssim\kappa^{-1}$, $m\approx\log(\kappa)$
suffices to guarantee stability and 
quasi-optimality
properties,
provided that $\kappa^\alpha h\lesssim 1$ where $\alpha$ depends on the
stability and regularity of the continuous problem.
The conditions on $h$ are the same as for the standard $Q_1$ FEM on the
global fine scale
(e.g.\ $\kappa^{3/2}h\lesssim 1$ for stability \cite{Wu2014CIPFEM}
and $\kappa^2h\lesssim 1$ for quasi-optimality \cite{melenk_phd}
in the case of pure Robin boundary conditions on a convex domain).
      
\subsection{Remarks on Generalizations of the Method}

The present approach exploits additional structure in the mesh
and thereby drastically decreases the cost for the
computation of the test functions 
$(\widetilde\Lambda_z\,:\,z\in\N_H)$.
Indeed, \eqref{e:lambdacorrectorproblem} is translation-invariant
and, thus, the number of corrector problems to be solved is
determined by the number of patch configurations. This number is
typically much smaller than the number of elements in $\G_H$,
see Figure~\ref{f:configurations} for an illustration.

Some remarks on more general versions of the presented msPGFEM
are in order.

\paragraph{Element shapes.}
As Figure~\ref{f:configurations} illustrates, highly structured
meshes are desirable as they lead to a moderate number of patch problems.
The method presented in Subsection~\ref{ss:defMethod} considers,
for simplicity, a partition of the domain in parallelepipeds.
While in scattering problems the outer part of the boundary $\Gamma_R$
results from a truncation of the full space and, hence, the choice of a
simple geometry (e.g., a cube) is justified,
it is extremely important to guarantee an accurate representation of
more general scattering objects.
This requires
more general element shapes such as isoparametric elements or
partitions in bricks and simplices with first-order ansatz functions
on the reference cell (see \cite{Peterseim2014} for simplicial meshes).
The msPGFEM and its error analysis is also applicable to this situation.
The configurations at the boundary will then determine the number of
corrector problems.

\paragraph{Fine-scale grid.}
The present approach is based on a global fine-scale grid
$\G_h$ and a particular choice of the domains $\Omega_T$,
which is convenient for the implementation of the method.
It is, however, not necessary for the domains $\Omega_T$ to be 
aligned with the mesh $\G_H$. Also the spaces $W_h(\Omega_T)$
can be defined over independent fine-scale meshes over $\Omega_T$.

\paragraph{Adaptive methods.}
For certain configurations of the domains $\Omega_T$, for instance in
the presence of re-entrant corners, it may be desirable to utilize
an adaptive fine-scale mesh over $\Omega_T$ for the solution of
the corrector problem \eqref{e:lambdacorrectorproblem}.
As proven in Lemma~\ref{l:wellposedideal} below, the corrector problems
are coercive and mesh-adaptation may improve the efficiency of the
fine-scale corrector problem.
As mentioned in the previous remark, it is indeed possible to
employ independent fine-scale meshes over different domains $\Omega_T$,
$\Omega_K$. 
The stability and error analysis for the adaptive case,
which are expected to be more involved,
are not discussed in this paper.

\begin{figure}
\begin{center}
\includegraphics[width=0.32\textwidth]{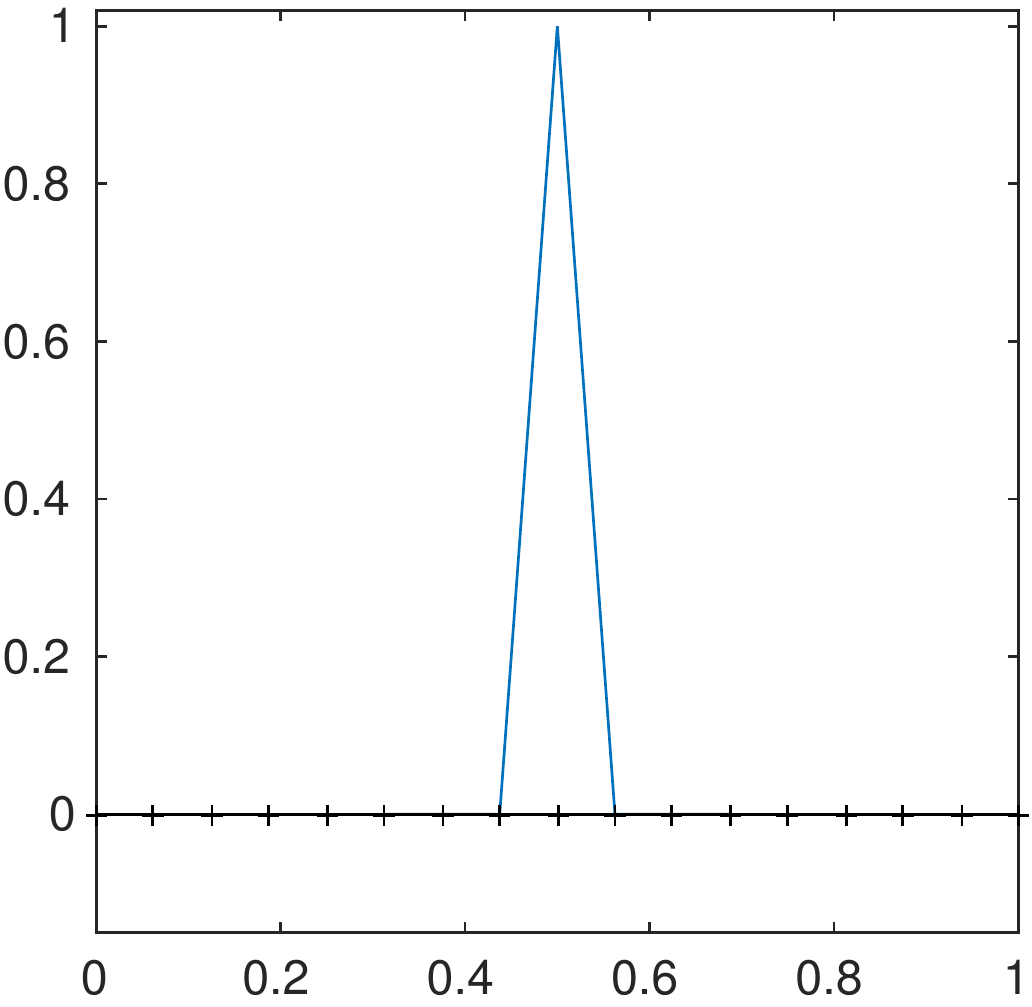}
\includegraphics[width=0.32\textwidth]{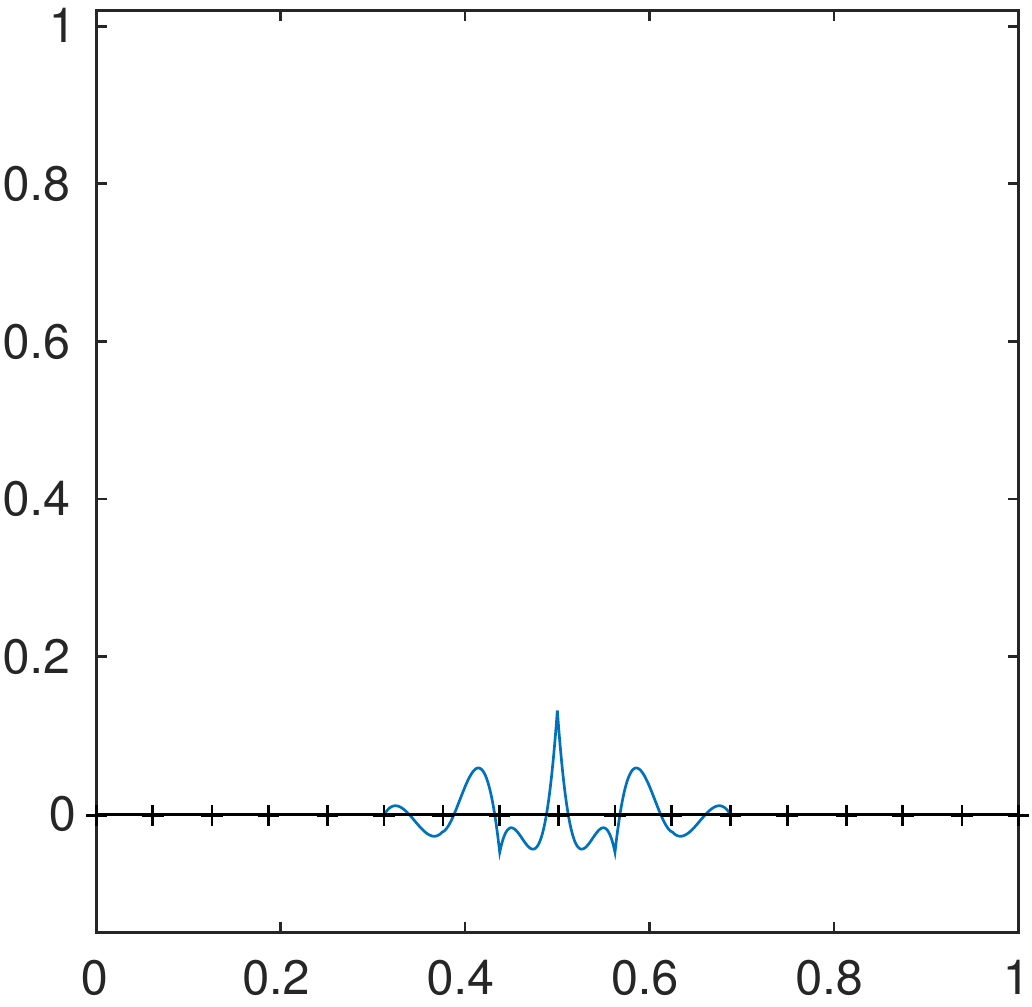}
\includegraphics[width=0.32\textwidth]{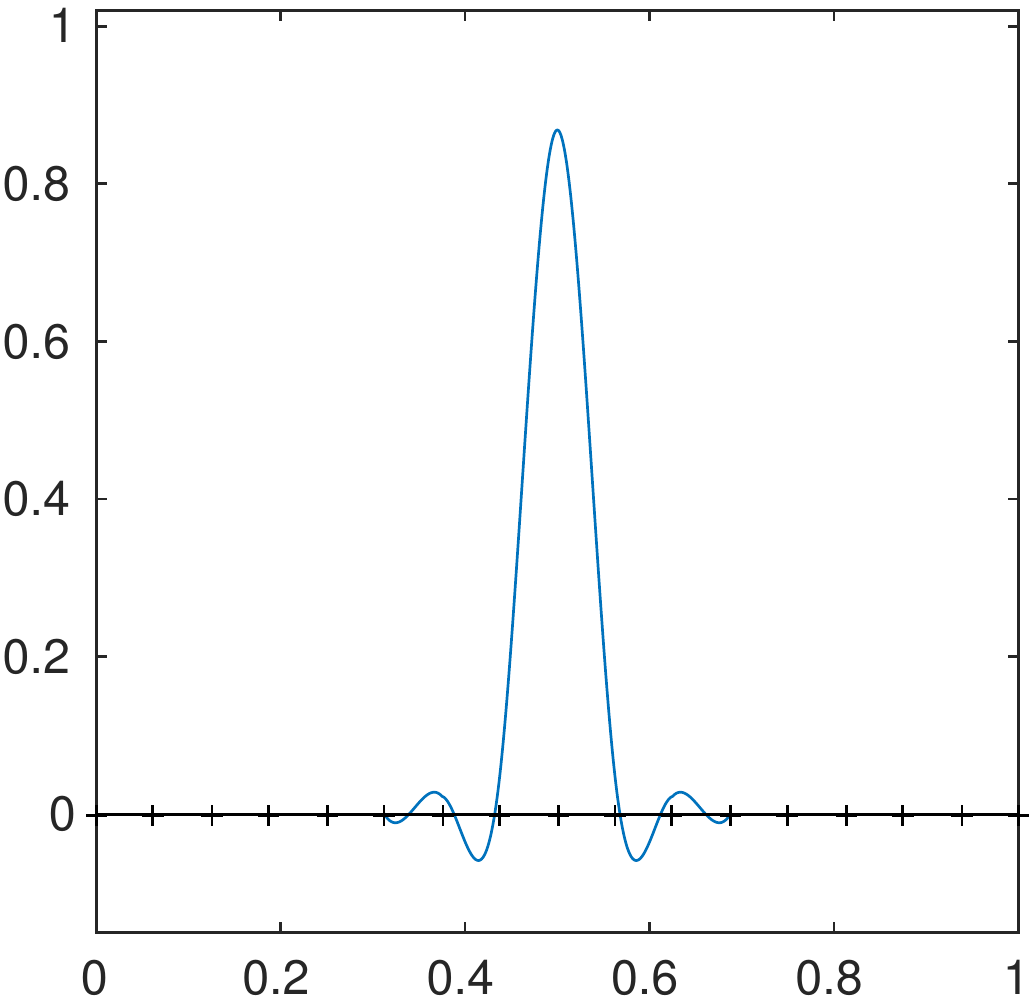}
\end{center}
\caption{Coarse-scale trial function $\Lambda_z$ (left),
         corrector $\lambda_z$ (middle),
         and modified test function 
         $\widetilde\Lambda_z=\Lambda_z-\lambda_z$ (right)
         in 1D
         with $\kappa=2^5$, $H=2^{-4}$, $h=2^{-10}$, $m=2$.
         \label{f:corrector1D}
         }
\end{figure}

\begin{figure}
\begin{center}
\includegraphics[width=0.49\textwidth]{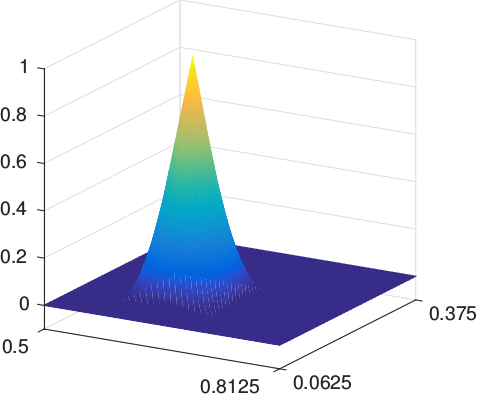}
\includegraphics[width=0.49\textwidth]{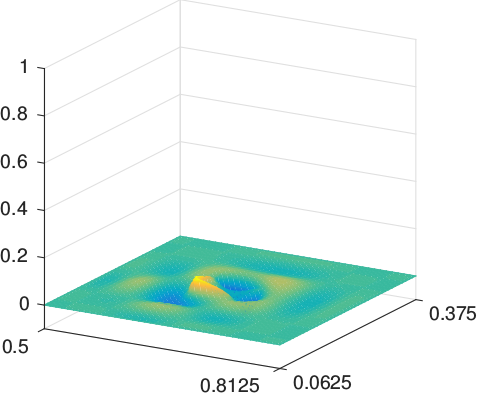}
\end{center}
\caption{Coarse-scale trial function $\Lambda_z$ (left),
         and element corrector $\lambda_{z,T}$ (right)
         in 2D
         with $\kappa=2^5$, $H=2^{-4}$, $h=2^{-7}$, $m=2$
         for the patch highlighted in Figure~\ref{f:configurations}.
         \label{f:corrector2D}
         }
\end{figure}

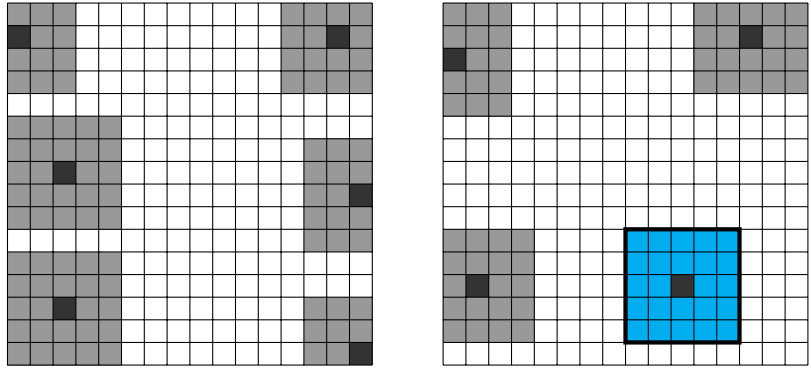
\begin{figure}
\begin{center}
\begin{tikzpicture}[scale=0.3]
\fill[black!40!white] (0,0) rectangle (5,5);
\fill[black!80!white] (2,2) rectangle (3,3);

\fill[black!40!white] (0,6) rectangle (5,11);
\fill[black!80!white] (2,8) rectangle (3,9);

\fill[black!40!white] (13,0) rectangle (16,3);
\fill[black!80!white] (15,0) rectangle (16,1);

\fill[black!40!white] (12,12) rectangle (16,16);
\fill[black!80!white] (14,14) rectangle (15,15);

\fill[black!40!white] (13,5) rectangle (16,10);
\fill[black!80!white] (15,7) rectangle (16,8);

\fill[black!40!white] (0,12) rectangle (3,16);
\fill[black!80!white] (0,14) rectangle (1,15);

\draw[step=1.0,black,thin] (0,0) grid (16,16);
\end{tikzpicture}
\hspace{4ex}
\begin{tikzpicture}[scale=0.3]
\fill[cyan] (8,1) rectangle (13,6);
\draw[black,ultra thick] (8,1) rectangle (13,6);
\fill[black!80!white] (10,3) rectangle (11,4);

\fill[black!40!white] (0,1) rectangle (4,6);
\fill[black!80!white] (1,3) rectangle (2,4);

\fill[black!40!white] (0,11) rectangle (3,16);
\fill[black!80!white] (0,13) rectangle (1,14);

\fill[black!40!white] (11,12) rectangle (16,16);
\fill[black!80!white] (13,14) rectangle (14,15);

\draw[step=1.0,black,thin] (0,0) grid (16,16);
\end{tikzpicture}
\end{center}
\caption{All possible patch configurations (up to rotations)
         on a structured mesh of
         the square domain with pure Robin boundary with
         $m=2$.
         A trial function and corresponding corrector for the highlighted
         patch is depicted in Figure~\ref{f:corrector2D}.
         \label{f:configurations}
         }
\end{figure}

\begin{figure}
\begin{center}
\includegraphics[width=.49\textwidth]{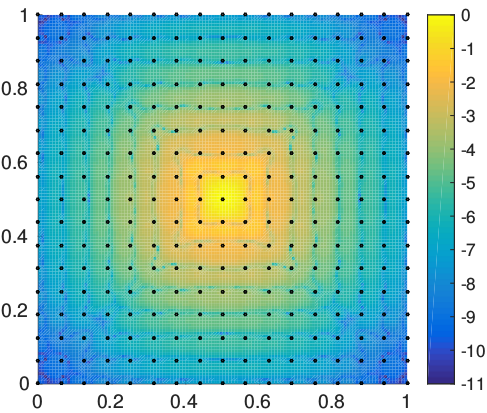}
\includegraphics[width=.49\textwidth]{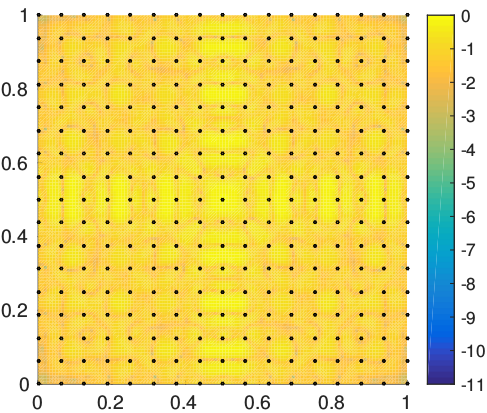}
\end{center}
\caption{Modulus of the idealized test function 
         $\widetilde\Lambda_z$ for $m=\infty$, $H=2^{-4}$, $h=2^{-7}$ in 2D
         in a logarithmically scaled plot.
         The dots indicate the grid points of the coarse mesh.
         Left: $\kappa=2^5$; right: $\kappa=2^6$.
         \label{f:testfct2Ddensity}
         }
\end{figure}

\section{Error Analysis}\label{s:erroranalysis}

We denote the global finite element space on the fine scale by
$V_h:=V_h(\Omega)=\cS^1(\G_h)\cap V$.
We denote the solution operator of the element corrector problem 
\eqref{e:lambdacorrectorproblem} by $\Cor_{T,m}$. 
Then any $z\in\N_H$ and any $T\in\G_H$ satisfy
$\lambda_{z,T} = \Cor_{T,m}(\Lambda_z)$ and we refer to
$\Cor_{T,m}$ as element correction operator.
The map $\Lambda_z\mapsto \lambda_z$
described in Subsection~\ref{ss:defMethod} defines a linear operator
$\Cor_m$ via $\Cor_m(\Lambda_z)=\lambda_z$ for any $z\in\N_H$,
referred to as correction operator.
For the analysis we introduce idealized counterparts of these
correction operators where the patch $\Omega_T$ equals $\Omega$.
Define the null space space
$W_h := \{ v\in V_h\,:\, I_H(v) = 0\}$.
For any $v\in V$, the idealized element corrector problem seeks
$\Cor_{T,\infty} v\in W_h$ such that
\begin{equation}\label{e:idealElementCorrProb}
a(w,\Cor_{T,\infty} v) = a_T(w,v)\quad\text{for all }w\in W_h.
\end{equation}
Furthermore, define
\begin{equation}\label{e:idealCorrector}
\Cor_\infty v:=\sum_{T\in\G_H} \Cor_{T,\infty} v.
\end{equation}

It is proved in \cite[Corollary~3.2]{MS10} that
the form $a$ is continuous and there is a
constant $C_a$ such that
\begin{equation*}
a(v,w) \leq C_a \|v\|_V \|w\|_V
\quad\text{for all } v,w\in V.
\end{equation*}
The following result implies the well-posedness of the 
idealized
corrector problems.

\begin{lemma}[well-posedness of the 
idealized
corrector problems]
\label{l:wellposedideal}
Provided 
\begin{equation}\label{e:resolution}
\CIH\sqrt{\Col} H\kappa \leq 1/\sqrt2,
\end{equation}
we have for all $w\in W_h$ equivalence of norms
\begin{equation*} 
\|\nabla w\|_{L^2(\Omega)} 
\leq \|w\|_V 
\leq \sqrt{3/2}\, \|\nabla w\|_{L^2(\Omega)}
\end{equation*}
and ellipticity
\begin{equation*} 
\frac12 \|\nabla w\|_{L^2(\Omega)}^2 \leq \Re a(w,w) .
\end{equation*}
\end{lemma}
\begin{proof}
For any $w\in W_h$ the property \eqref{e:IHapproxstab} implies
\begin{equation*}
\kappa^2 \|w\|_{L^2(\Omega)}^2 
= \kappa^2 \|(1-I_H)w\|_{L^2(\Omega)}^2
\leq
\CIH^2\Col H^2\kappa^2 \|\nabla w\|_{L^2(\Omega)}^2.
\qedhere
\end{equation*}
\end{proof}

Lemma~\ref{l:wellposedideal} implies that the idealized corrector
problems \eqref{e:idealCorrector}
are well-posed and the correction operator $\Cor_\infty$
is continuous in the sense that 
\begin{equation*}
\|\Cor_\infty v_H \|_V \leq C_\Cor \|v_H\|_V
\quad\text{for all } v_H \in V_H
\end{equation*}
for some constant $C_\Cor \approx 1$.
Since 
the inclusion $W_h(\Omega_T)\subseteq W_h$ holds,
the well-posedness result of Lemma~\ref{l:wellposedideal} carries over
to the corrector problems \eqref{e:lambdacorrectorproblem}
in the subspace $W_h(\Omega_T)$
with the sesquilinear form $a_{\Omega_T}$.

The proof of well-posedness of the Petrov-Galerkin method
\eqref{e:discreteproblem} will be
based on the fact that the difference $(\Cor_\infty-\Cor_m)(v)$
decays exponentially with the distance from $\supp(v)$.
In the next theorem, we quantify the difference between the
idealized and the discrete correctors.
The proof will be given in Appendix~\ref{a:proof} of this paper
and is based on the exponential decay of the corrector $\Cor_\infty \Lambda_z$
itself, see Figure~\ref{f:testfct2Ddensity}.
That figure also illustrates that the decay requires the resolution
condition \eqref{e:resolution}, namely $\kappa H \lesssim 1$.

\begin{theorem}\label{t:CorrCloseness}
Under the resolution condition \eqref{e:resolution}
there exist constants 
$C_1\approx 1 \approx C_2$ and $0<\beta<1$ such that
any $v\in V_H$, any $T\in\G_H$ and any $m\in\mathbb N$ satisfy
\begin{align}
\label{e:CorrCloseness1}
\|\nabla(\Cor_{T,\infty} v - \Cor_{T,m} v)\|_{L^2(\Omega)}
&
\leq C_1
 \beta^m \|\nabla v \|_{L^2(T)},
\\
\label{e:CorrCloseness2}
\|\nabla(\Cor_\infty v - \Cor_m v)\|_{L^2(\Omega)}
&
\leq C_2 \sqrt{\Colm}
\beta^m \|\nabla v \|_{L^2(\Omega)}.
\end{align}
\qed
\end{theorem}

Provided $h$ is chosen fine enough,
the standard FEM over $\G_h$ is stable in the sense that
there exists a constant $C_{\mathrm{FEM}}$
such that with $\gamma(\kappa,\Omega)$ 
from \eqref{e:helmholtzstability} 
there holds
\begin{equation}\label{e:finehstability}
 \big(C_{\mathrm{FEM}}\gamma(\kappa,\Omega) \big)^{-1}
  \leq
  \inf_{v\in V_h\setminus\{0\}}\sup_{w\in V_h\setminus\{0\}} 
      \frac{\Re a(v,w)}{\|v\|_V\|w\|_V} .
\end{equation}
This is actually a condition on the fine-scale parameter $h$.
In general, the requirements on $h$ depend on the stability of the 
continuous problem \cite{melenk_phd}.

\begin{theorem}[well-posedness of the discrete problem]\label{t:wellposed}
Under the resolution conditions \eqref{e:resolution}
and \eqref{e:finehstability}
and the following oversampling condition
\begin{equation}\label{e:mcondition}
m\geq 
\lvert\log\big(\sqrt{6} C_a \sqrt{\Col} \CIH C_{I_H,V} C_2 \sqrt{\Colm} C_{\mathrm{FEM}}\gamma(\kappa,\Omega)\big)\rvert
\Big/\lvert \log(\beta)\rvert,
\end{equation}
problem \eqref{e:discreteproblem} is well-posed and the constant
$C_{\mathrm{PG}}:=2C_{I_H,V}C_\Cor C_{\mathrm{FEM}}$ satisfies
\begin{equation*}
\big(C_{\mathrm{PG}}\gamma(\kappa,\Omega)\big)^{-1}
\leq
 \inf_{v_H\in V_H\setminus\{0\}}
 \sup_{\tilde v_H\in \widetilde V_H\setminus\{0\}}
 \frac{\Re a(v_H,\tilde v_H)}{\|v_H\|_V\|\tilde v_H\|_V} .
\end{equation*}
\end{theorem}
\begin{proof}
Let $u_H\in V_H$ with $\|u_H\|_V =1$.
From \eqref{e:finehstability} we infer that
there exists some $v\in V_h$ with $\|v\|_V =1$ such that
\begin{equation*}
\Re a(u_H -\overline{\Cor_\infty(\bar u_H)},v) 
\geq \big(C_{\mathrm{FEM}}\gamma(\kappa,\Omega) \big)^{-1} \|u_H -\overline{\Cor_\infty(\bar u_H)}\|_V  .
\end{equation*}
It follows from the structure of the sesquilinear form $a$ that
$\overline{\Cor_\infty(\bar u_H)}$ solves the following adjoint
corrector problem
\begin{equation}\label{e:dualCor}
a(\overline{\Cor_\infty(\bar u_H)},w) = a(u_H,w)
\quad\text{for all }w\in W_h,
\end{equation}
cf.\ \cite[Lemma~3.1]{mm_stas_helm3}.
Let $\tilde v_H := (1-\Cor_m) I_H v \in \widetilde V_H$.
We have
\begin{equation} \label{e:stabilityTwoTerms}
a(u_H,\tilde v_H) 
=
a(u_H,(1-\Cor_\infty) I_H v) 
+
a(u_H,(\Cor_\infty-\Cor_m) I_H v) .
\end{equation}
Since $\Cor_\infty$ is a projection onto $W_h$,
we have $(1-\Cor_\infty)(1-I_H) v =0$
and, thus,
$(1-\Cor_\infty)I_H v =(1-\Cor_\infty) v$.
The solution properties \eqref{e:dualCor} of
$\overline{\Cor_\infty(\bar u_H)}$ 
and \eqref{e:idealElementCorrProb}--\eqref{e:idealCorrector}
of $\Cor_\infty v$ prove
$a(u_H,\Cor_\infty v) = a(\overline{\Cor_\infty(\bar u_H)},v)$.
Hence,
\begin{equation*}
\begin{aligned}
\Re a(u_H,(1-\Cor_\infty) I_H v)
&
= \Re a(u_H -\overline{\Cor_\infty(\bar u_H)},v) 
\\
&
\geq
\big(C_{\mathrm{FEM}}\gamma(\kappa,\Omega) \big)^{-1} 
    \|u_H -\overline{\Cor_\infty(\bar u_H)}\|_V .
\end{aligned}
\end{equation*}
Furthermore, the estimate \eqref{e:IHapproxstabV} implies
\begin{equation*}
1=\|u_H\|_V = \|I_H (u_H -\overline{\Cor_\infty(\bar u_H)}) \|_V
\leq
C_{I_H,V}
\|u_H -\overline{\Cor_\infty(\bar u_H)}\|_V .
\end{equation*}
The second term on the right-hand side of \eqref{e:stabilityTwoTerms}
satisfies with $\|u_H\|_V = 1$ and Lemma~\ref{l:wellposedideal} that
\begin{equation*}
\lvert a(u_H,(\Cor_\infty-\Cor_m) I_H v) \rvert
\leq \sqrt{3/2} C_a \|\nabla (\Cor_\infty-\Cor_m) I_H v\|_{L^2(\Omega)}.
\end{equation*}
Altogether, it follows  that
\begin{equation*}
\Re a(u_H,\tilde v_H)
\geq 
\left(
\frac{1}{C_{I_H,V} C_{\mathrm{FEM}}\gamma(\kappa,\Omega) }
-
 \sqrt{\frac32} C_a
\|\nabla(\Cor_\infty-\Cor_m) I_H v \|_{L^2(\Omega)}
\right)
.
\end{equation*}
Theorem~\ref{t:CorrCloseness} and \eqref{e:IHapproxstab} show that
\begin{equation*}
\|\nabla(\Cor_\infty-\Cor_m) I_H v \|_{L^2(\Omega)}
\leq
   C_2 \sqrt{\Colm}\beta^m \|\nabla I_H v\|
\leq
   C_2 \sqrt{\Colm}\CIH\sqrt{\Col}\beta^m .
\end{equation*}
Hence, the condition \eqref{e:mcondition} and
$\| \tilde v_H \|_V = \|(1-\Cor_\infty) v \|_V \leq C_\Cor$
imply the assertion.
\end{proof}

\begin{remark}[adjoint problem]\label{r:adjoint}
Under the assumptions of Theorem~\ref{t:wellposed},
problem \eqref{e:discreteproblem} is well-posed and, thus, it 
follows from a dimension argument that there is  non-degeneracy of the
sesquilinear form $a$ over $V_H\times\widetilde V_H$. 
Thus, the adjoint problem to \eqref{e:discreteproblem}
is well-posed with the same stability constant as in 
Theorem~\ref{t:wellposed}.
\end{remark}

The quasi-optimality result requires the following additional condition
on the oversampling parameter $m$,
\begin{equation}\label{e:oversampling2}
m
\geq
\lvert
\log\Big( 
2 C_2 \sqrt{\Colm} C_a^2 C_{\mathrm{PG}}\gamma(\kappa,\Omega) \sqrt{3/2}
\Big)
\rvert
\Big/ \lvert \log(\beta) \rvert.
\end{equation}

\begin{theorem}[quasi-optimality]\label{t:quasiopt}
 The resolution conditions \eqref{e:resolution} and
 \eqref{e:finehstability}
 and the oversampling
 conditions \eqref{e:mcondition} and 
 \eqref{e:oversampling2}
 imply that the solution $u_H$ to
 \eqref{e:discreteproblem} with parameters $H$, $h$, and $m$ 
 and the solution $u_h$ of the standard Galerkin FEM on the mesh
 $\G_h$ satisfy
\begin{equation*}
\|u_h - u_H\|_V
\lesssim
  \|(1-I_H) u_h \|_V
\approx 
  \min_{v_H\in V_H} \| u_h - v_H \|_V .
\end{equation*}
\end{theorem}

\begin{proof}
Let $e:=u_h-u_H$.
The triangle inequality and Lemma~\ref{l:wellposedideal} yield
\begin{equation*}
\|e\|_V\leq \|(1-I_H) u_h \|_V + \|I_H e\|_V .
\end{equation*}
It remains to bound the second term on the right-hand side.
The proof employs a standard duality argument, the stability of
the idealized method and the fact that our practical method is a
perturbation of that ideal method.
Let $z_H \in V_H$ be the solution to the dual problem
\begin{equation*}
(\nabla v_H, \nabla I_H e) + \kappa^2(v_H,I_H e) = a(v_H,(1-\Cor_\infty) z_H)
\end{equation*}
for all $v_H\in V_H$ (cf.\ Remark~\ref{r:adjoint}).
The choice of the test function $v_H= I_H e$ implies that
\begin{equation*}
\|I_H e\|_V^2
=
a(I_H e,(1-\Cor_\infty) z_H)
=
a(I_H e,(\Cor_m-\Cor_\infty) z_H) + a(I_H e,(1-\Cor_m) z_H).
\end{equation*}
The identity $I_H(\Cor_m-\Cor_\infty) z_H = 0$,
the resolution condition \eqref{e:resolution},
the estimate \eqref{e:CorrCloseness2},
and the stability of the adjoint problem imply
for the first term on the right-hand side that
\begin{equation*}
\begin{aligned}
&
a(I_H e,(\Cor_m-\Cor_\infty) z_H)
\\
&
\leq
C_a \sqrt{3/2}
\|I_H e\|_V \|\nabla(\Cor_m-\Cor_\infty) z_H\|_{L^2(\Omega)}
\\
&
\leq
C_2 \sqrt{\Colm} C_a \sqrt{3/2}
\|I_H e\|_V \beta^m \|\nabla z_H\|
\\
&
\leq
C_2 \sqrt{\Colm} C_a^2 C_{\mathrm{PG}}\gamma(\kappa,\Omega) \sqrt{3/2}
\beta^m \|I_H e\|_V^2.
\end{aligned}
\end{equation*}
The condition \eqref{e:oversampling2} implies that this is
$\leq \frac12 \|I_H e\|_V^2$.
The Galerkin orthogonality $a(u_h - u_H, (1-\Cor_m) z_H) =0$,
the solution property \eqref{e:idealCorrector} of $\Cor_\infty z_H$,
the resolution condition \eqref{e:resolution}
and the exponential decay \eqref{e:CorrCloseness2}
imply for the second term
\begin{equation*}
\begin{aligned}
a(I_H e,(1-\Cor_m) z_H) 
&= a(I_H u_h - u_h,(1-\Cor_m) z_H)
\\
&= a(I_H u_h - u_h,(\Cor_\infty-\Cor_m) z_H)
\\
&\leq \sqrt{3/2}C_a C_2 \sqrt{\Colm} \beta^m \|I_H u_h - u_h\|_V \|\nabla z_H\|_{L^2(\Omega)}
.
\end{aligned}
\end{equation*}
The stability of the adjoint problem implies
\begin{equation*}
\|\nabla z_H\|_{L^2(\Omega)} \leq C_{\mathrm{PG}} \gamma(\kappa,\Omega)C_a \|I_H e\|_V.
\end{equation*}
Thus,
\begin{equation*}
a(I_H e,(1-\Cor_m) z_H)
\leq 
\sqrt{3/2}C_a^2 C_2 \sqrt{\Colm} C_{\mathrm{PG}} \beta^m \gamma(\kappa,\Omega)
\|I_H u_h - u_h\|_V  \|I_H e\|_V
.
\end{equation*}
The term $\|I_H e\|_V$ can be absorbed and the oversampling condition
\eqref{e:mcondition} implies that $\beta^m \sqrt{\Colm} \gamma(\kappa,\Omega)$
is controlled by some $\kappa$-independent constant.
The combination with the 
foregoing displayed formulae concludes the proof.
\end{proof}

The following consequence of Theorem~\ref{t:quasiopt} states an
estimate for the error $u-u_H$.

\begin{corollary}
Under the conditions of Theorem~\ref{t:quasiopt}, the discrete
solution $u_H$ to \eqref{e:discreteproblem} satisfies with some
constant $C\approx 1$ that
\begin{equation*}
\| u- u_H \|_V
\leq
  \| u- u_h \|_V
  +
  C \min_{v_H\in V_H} \| u_h - v_H \|_V .
\end{equation*}
In particular, provided that the solution satisfies
$u\in W^{1,s}(\Omega)$ for $0<s\leq 1$, the error decays as
$\| u- u_H \|_V\leq\mathcal{O}(H^s)$.\qed
\end{corollary}

\begin{remark}\label{r:IHuhformula}
In the idealized case that $m=\infty$, we have 
$u_h - I_H u_h \in W_h$ and, thus,
\begin{equation*}
a(u_h - I_H u_h, (1-\Cor_\infty) v_H) = 0
\quad\text{for all } v_H\in V_H.
\end{equation*}
Therefore, problem \eqref{e:discreteproblem} and the Galerkin
property show that $u_H = I_H u_h$.
\end{remark}

\section{Numerical Experiments}\label{s:num}

We investigate the method in three numerical experiments.
The convergence history plots display the absolute error in the norm
$\|\cdot\|_V$ versus the mesh size $H$.

\subsection{Plane Wave on the Square Domain}\label{ss:2dPlaneWave}
On the unit square $\Omega=(0,1)^2$, we consider the pure Robin problem 
$\Gamma_R=\partial\Omega$ 
with data given by the plane wave
$u(x)=\exp(-i\kappa x\cdot
 \left(\begin{smallmatrix}0.6\\0.8\end{smallmatrix}\right))$.

\begin{figure}
\begin{center}
\subfloat[\label{sf:PlaneWave2Dkappadependence64}]{
\includegraphics[width=.49\textwidth]{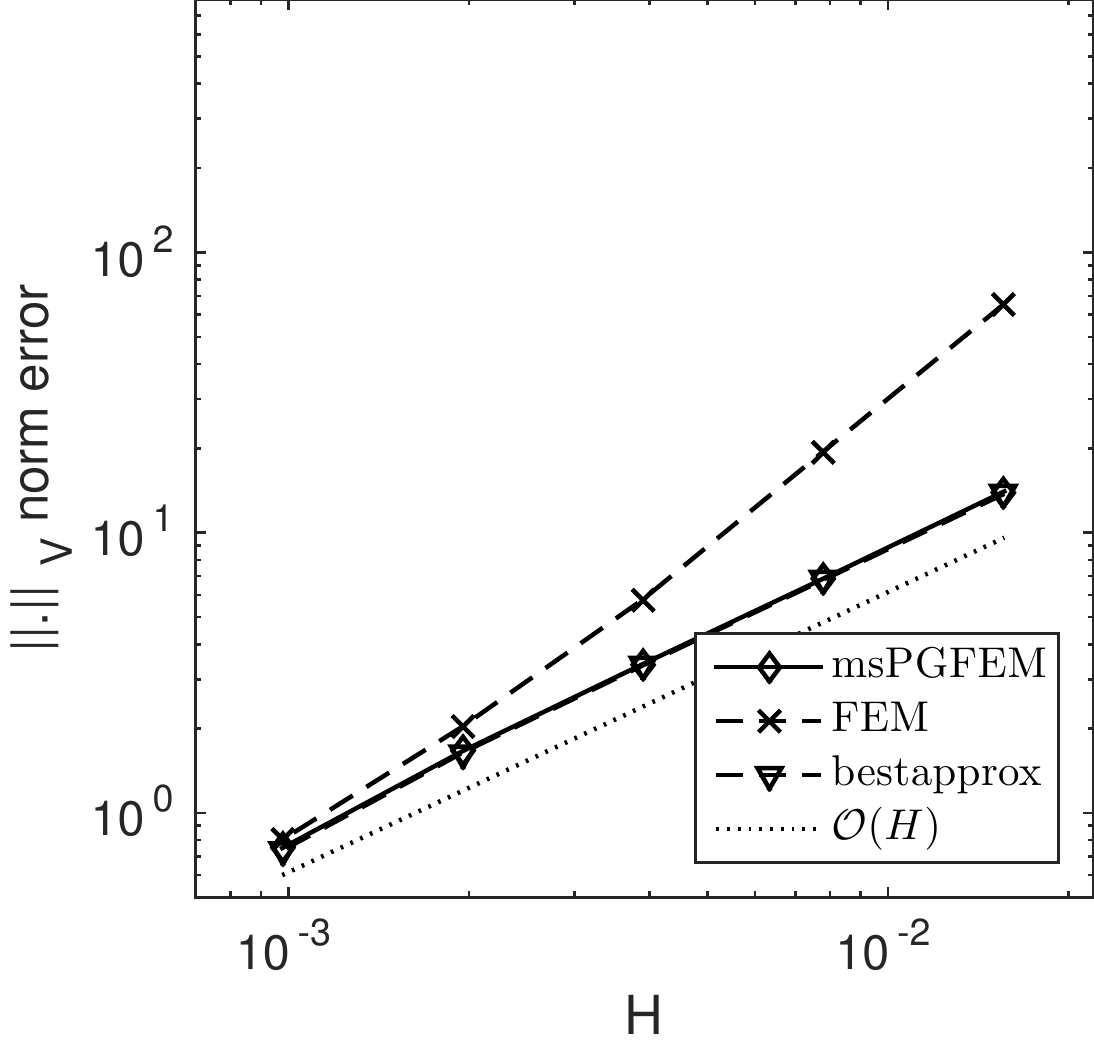}
}
\subfloat[\label{sf:PlaneWave2Dkappadependence128}]{
\includegraphics[width=.49\textwidth]{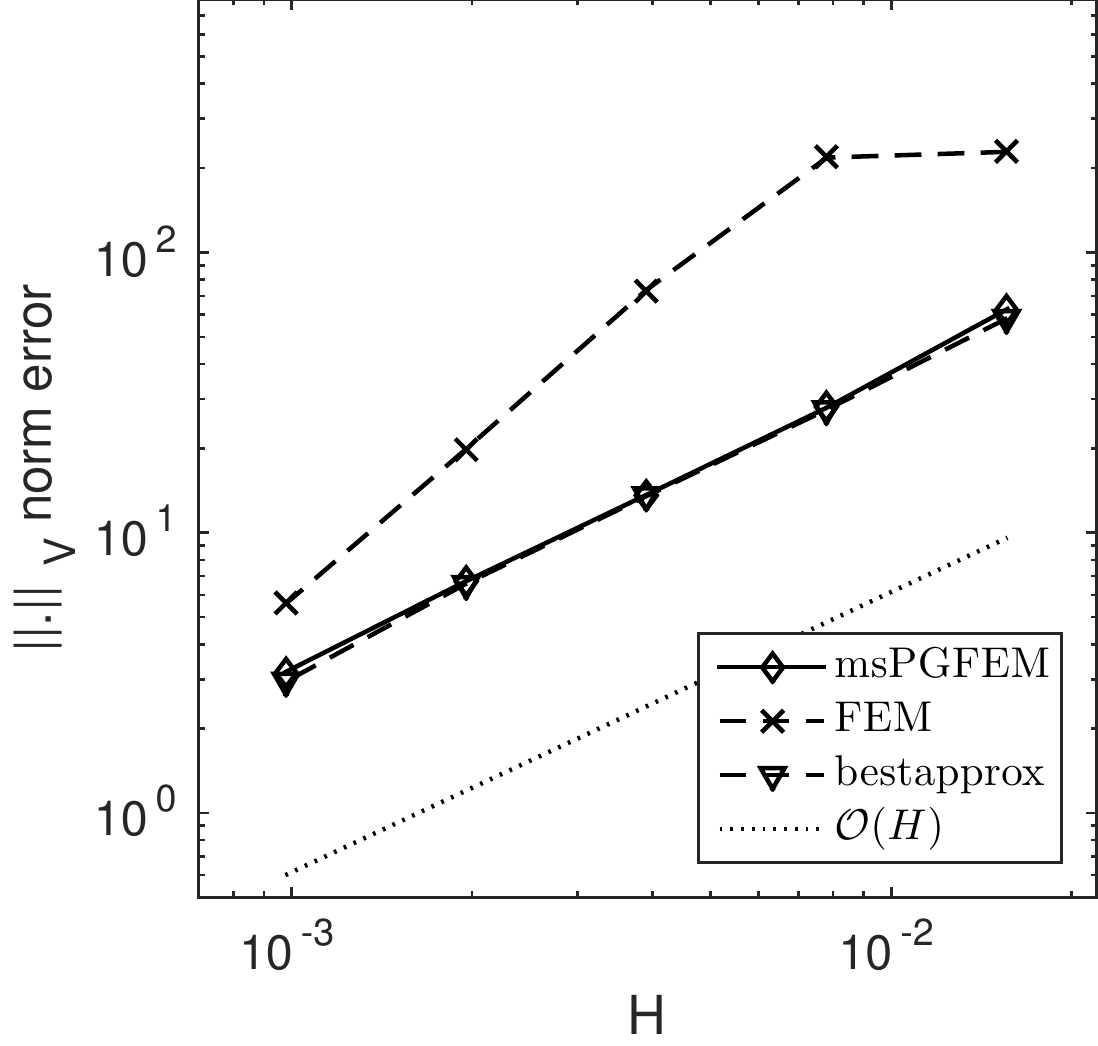}
}
\\
\subfloat[\label{sf:PlaneWave2Dkappadependence256}]{
\includegraphics[width=.49\textwidth]{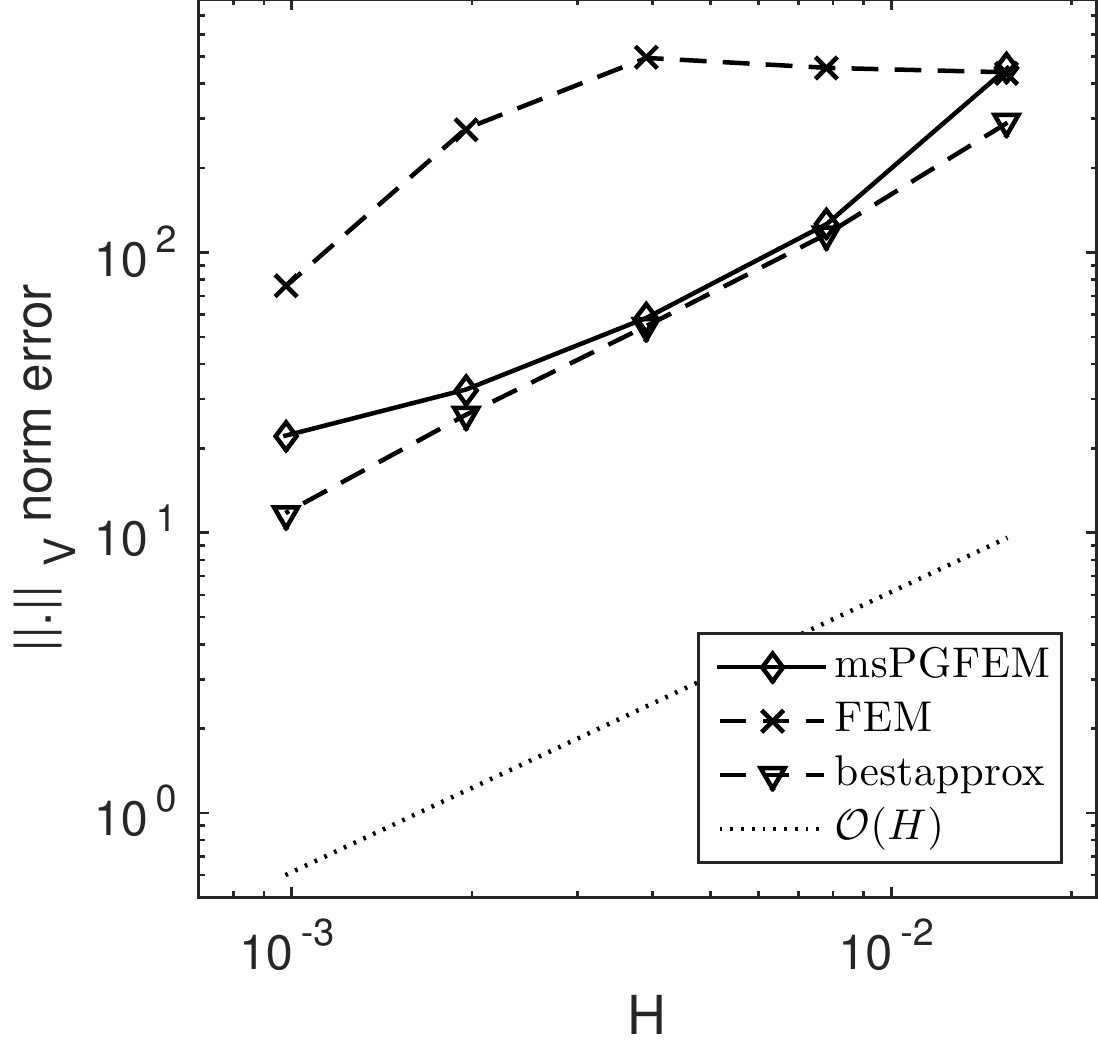}
}
\subfloat[\label{sf:PlaneWave2Dhdependence}]{
\includegraphics[width=.49\textwidth]{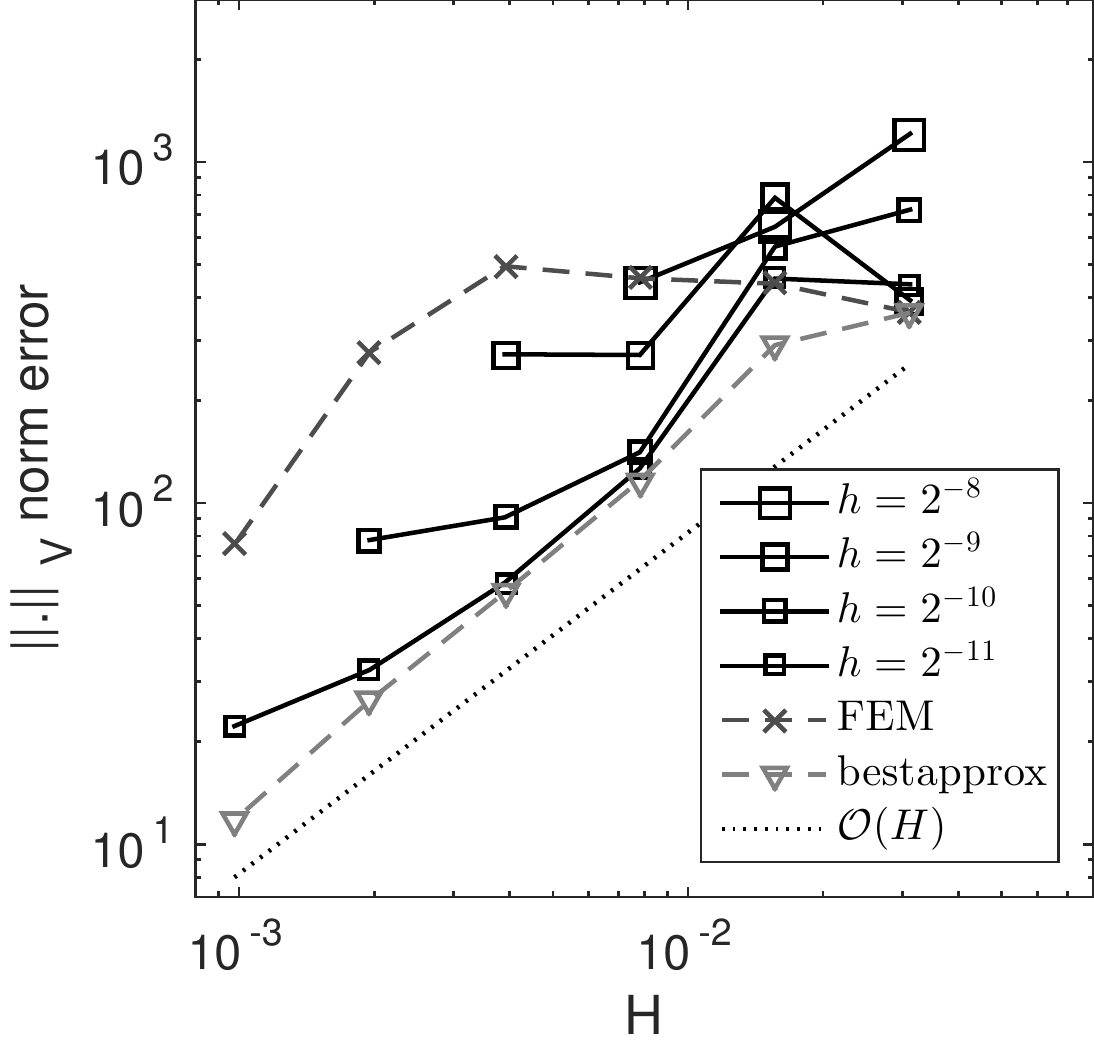}
}
\caption{(a)--(c). Comparison of the msPGFEM with the  best approximation in
         $\|\cdot\|_V$ and the standard Galerkin FEM
         for the 2D plane wave example for $\kappa=2^6,2^7,2^8$.
         (d). Dependence on the fine mesh parameter $h$
         in the 2D plane wave example with $\kappa=2^8$.
         }   
\end{center}
\end{figure}
 
Figure~\ref{sf:PlaneWave2Dkappadependence64}--\ref{sf:PlaneWave2Dkappadependence256}
displays the convergence
history for $\kappa=2^6,2^7,2^8$ and the fine-scale mesh parameter
$h=2^{-11}$.
The best-approximation error of continuous $Q_1$ functions
in $\|\cdot\|_V$ and the error of
the standard Galerkin FEM on the same coarse mesh are plotted
for comparison.
As expected, the standard FEM clearly exhibits the pollution effect,
and larger values of $\kappa$ increase the discrepancy between the
approximation error of the FEM and the theoretical best-approximation
by $Q_1$ functions in the regime under consideration.
In contrast, the approximation by the msPGFEM  can compete with the 
best-approximation on meshes that allow a meaningful representation of
the solution.
We stress the fact that the convergence history plots merely
take into account the coarse mesh-size $H$, but the computational cost
in the multiscale method is moderately higher than in the standard FEM
due to the increased communication caused by the coupling
$m\approx\log(\kappa)$.

For the oversampling parameter $m=2$, the number of corrector problems
to be solved 
for the finest mesh $\G_H$
is 49 out of $1\,048\,576$
when no symmetry is exploited.

Figure~\ref{sf:PlaneWave2Dhdependence}  displays
the dependence on the fine mesh parameter $h$ for
$\kappa=2^8$ and oversampling parameter $m=6$.
Since the multiscale method based on the fine grid $\G_h$ computes
approximations of the FEM solution on that fine grid, 
e.g.\ $u_H = I_H u_h$ for $m=\infty$ as in Remark~\ref{r:IHuhformula},
it is clear that
the accuracy of the msPGFEM is limited by the accuracy of the standard
FEM on the fine scale. This can be observed in
Figure~\ref{sf:PlaneWave2Dhdependence}.
It can be also seen that a finer fine-scale mesh-size $h$ improves the
error of the msPGFEM towards the best-approximation.
In this two-dimensional example, 
the quasi-optimality constant appears to be close to 1

Next, we study the dependence on the oversampling parameter $m$.
Figure~\ref{f:PlaneWave2Dconvergence} displays the convergence
history for $\kappa=2^7$ and $\kappa=2^8$.
The fine mesh parameter is $h=2^{-11}$ and $m$ varies from
$m=1$ to $m=6$.
It turns out that for the present configuration, the value $m=2$ is
sufficient for quasi-optimality.
In the 
range
where $H$ is significantly
larger than $\kappa^{-1}$ and the resolution condition is violated,
larger oversampling parameters may lead
to larger errors,
which is not surprising in view of the lack of decay, see also
Figure~\ref{f:testfct2Ddensity}.
This, however, is no more the case as soon as
$H$ is small enough to allow for a meaningful representation of the 
wave.

\begin{figure}
\begin{center}
\includegraphics[width=.49\textwidth]{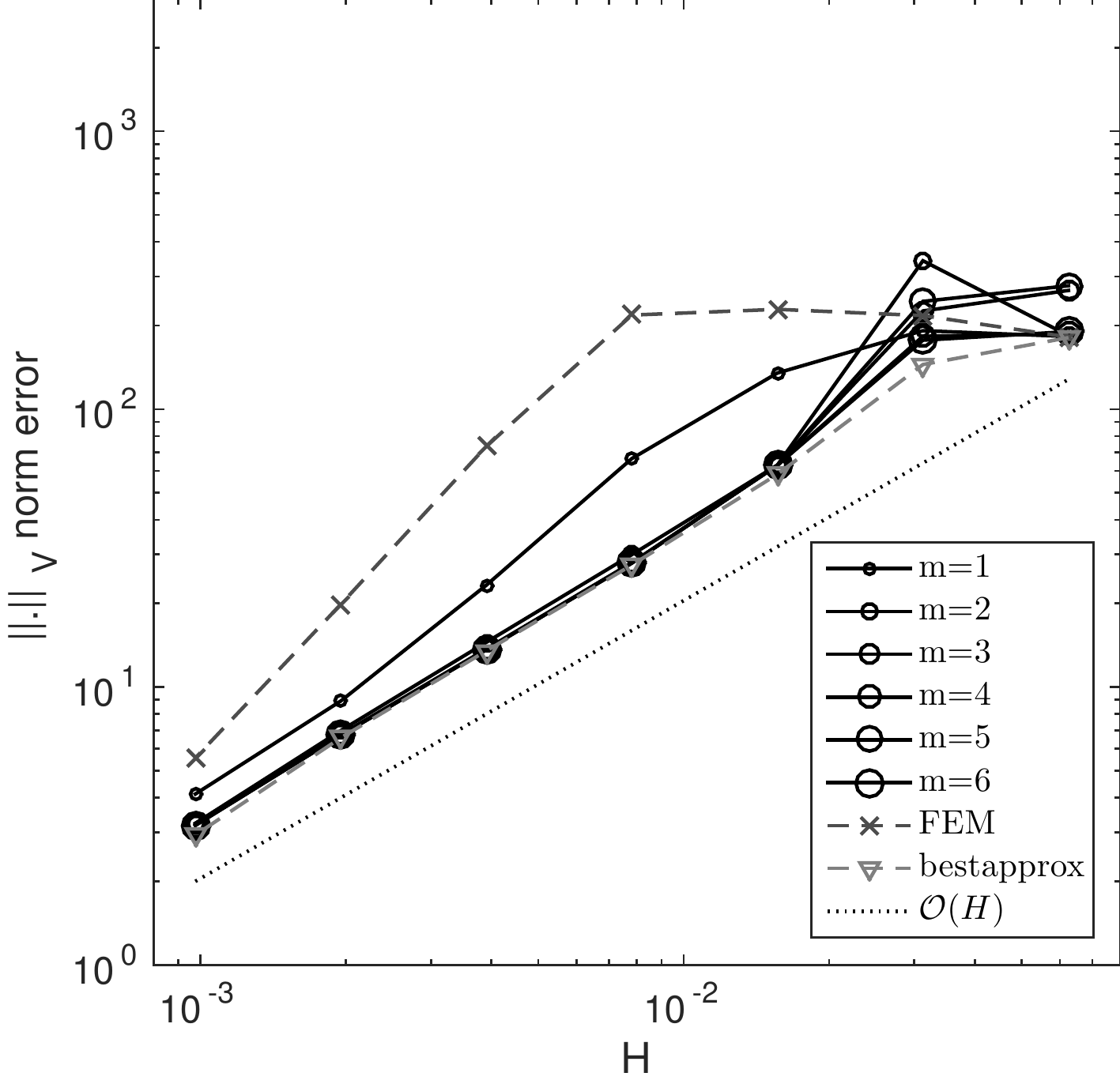}
\includegraphics[width=.49\textwidth]{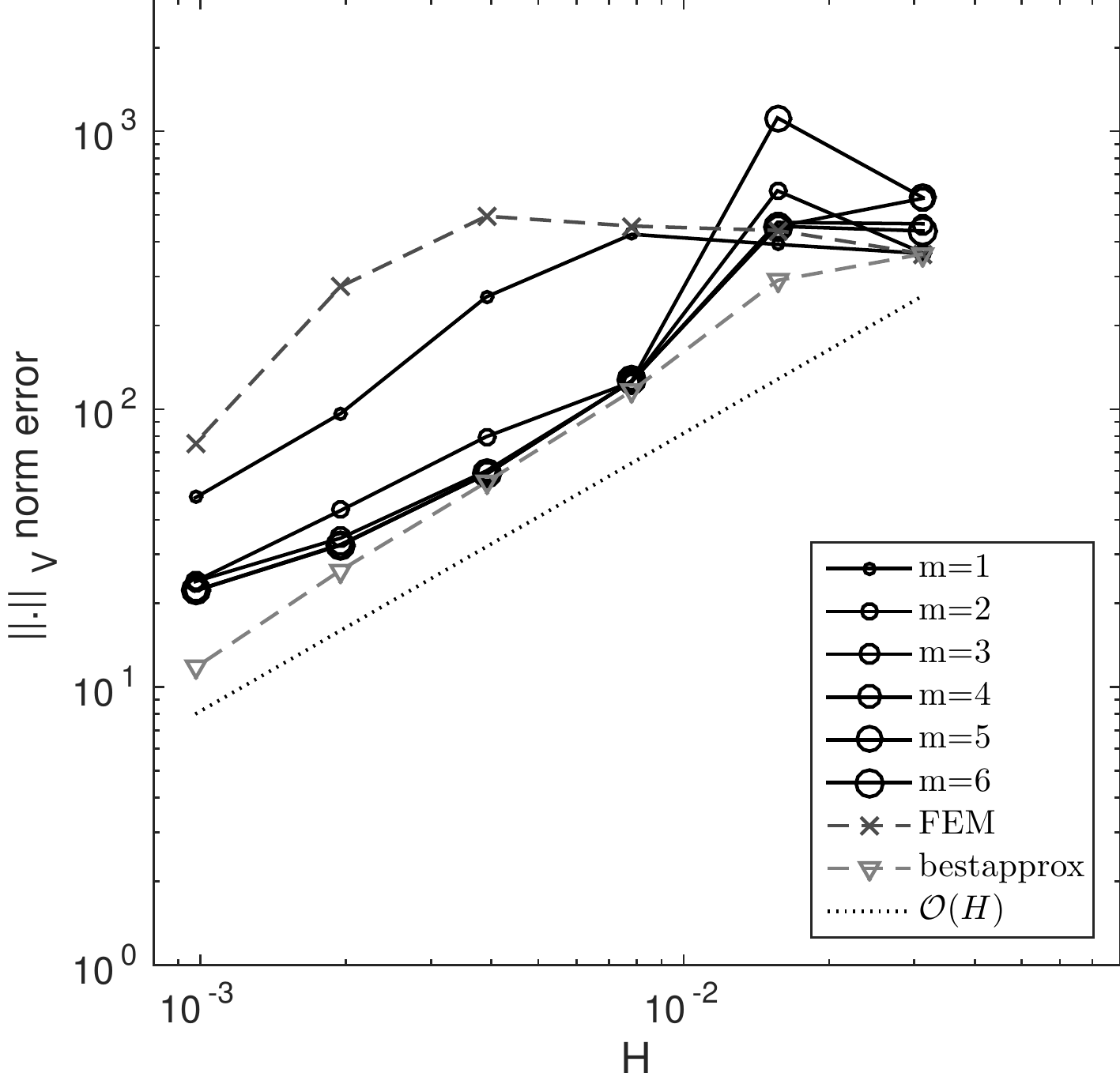}
\caption{Convergence history for the 2D plane wave example
         with $\kappa=2^7$ (left) and $\kappa=2^8$(right),
         $h=2^{-11}$ and varying $m$.
         \label{f:PlaneWave2Dconvergence}}
\end{center}
\end{figure}

\subsection{Multiple Sound-Soft Scatterers in 2D}\label{ss:multScatterer}
\begin{figure}
\begin{center}
\begin{tikzpicture}[scale=0.2]
\draw[step=1.0,black,thin] (0,0) grid (16,16);
\fill[black!90!white] (5,5) rectangle (7,7);
\draw[black!40!white,ultra thick] (5,5) rectangle (7,7);
\fill[black!90!white] (10,8) rectangle (12,10);
\draw[black!40!white,ultra thick] (10,8) rectangle (12,10);
\fill[black!90!white] (4,10) rectangle (6,13);
\draw[black!40!white,ultra thick] (4,10) rectangle  (6,13);
\end{tikzpicture}
\end{center}
\caption{Coarse mesh for the square domain with three scatterers
         from Subsection~\ref{ss:multScatterer}
         \label{f:scatteringdomain}
         }
\end{figure}
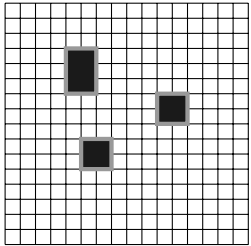

We consider the domain
\begin{equation*}
 \Omega := (0,1)^2 \big\backslash 
        \big( 
        \textstyle
          \left[\frac{5}{16},\frac{7}{16}\right]
            \times \left[\frac{5}{16},\frac{7}{16}\right]
           \cup
           \left[\frac{10}{16},\frac{12}{16}\right]
             \times \left[\frac{8}{16},\frac{10}{16}\right]
           \cup
           \left[\frac{4}{16},\frac{6}{16}\right]
             \times \left[\frac{10}{16},\frac{13}{16}\right]
        \big)
\end{equation*}
from Figure~\ref{f:scatteringdomain}.
The incident wave 
$u_{\mathrm{in}}(x)=\exp(-i\kappa x\cdot
 \left(\begin{smallmatrix}0.6\\0.8\end{smallmatrix}\right))$
is incorporated through the Robin boundary condition with
$g:= i\kappa u_{\mathrm{in}} + \partial_\nu u_{\mathrm{in}}$
on the outer boundary $\Gamma_R:=\{x\in \{0,1\} \text{ or } y\in\{0,1\}\}$.
On the remaining part of the
boundary $\Gamma_D:=\partial\Omega\setminus\Gamma_R$
we impose homogeneous Dirichlet conditions.
We choose the fine mesh parameter as $h=2^{-11}$.
Since the exact solution is unknown, we compute a reference solution
with the standard $Q_1$ FEM on the fine mesh $\G_h$
and we compare the coarse approximation with this reference solution.
Errors committed by the fine scale are not included in the discussion.
Figure~\ref{f:scatteringconv}
displays the convergence history for $\kappa=2^5$ 
and $\kappa=2^6$.
The oversampling parameter $m$ varies from $m=1$ to $m=4$.
As in the foregoing example, the value $m=2$ for the oversampling parameter
seems to be sufficient for the quasi-optimality and even a quasi-optimality
constant close to 1
in the range of wave numbers considered here.
In particular, the pollution effect that is visible
for the standard Galerkin FEM is not present for the msPGFEM.
Reduced convergences rates which are expected from the presence
of re-entrant corners are not visible in this computational range.
For the oversampling parameter $m=2$, the number of corrector problems
to be solved 
for the finest mesh $\G_H$
is 210 out of $61\,952$
when no symmetry is exploited.

\begin{figure}
\begin{center}
\includegraphics[width=.49\textwidth]{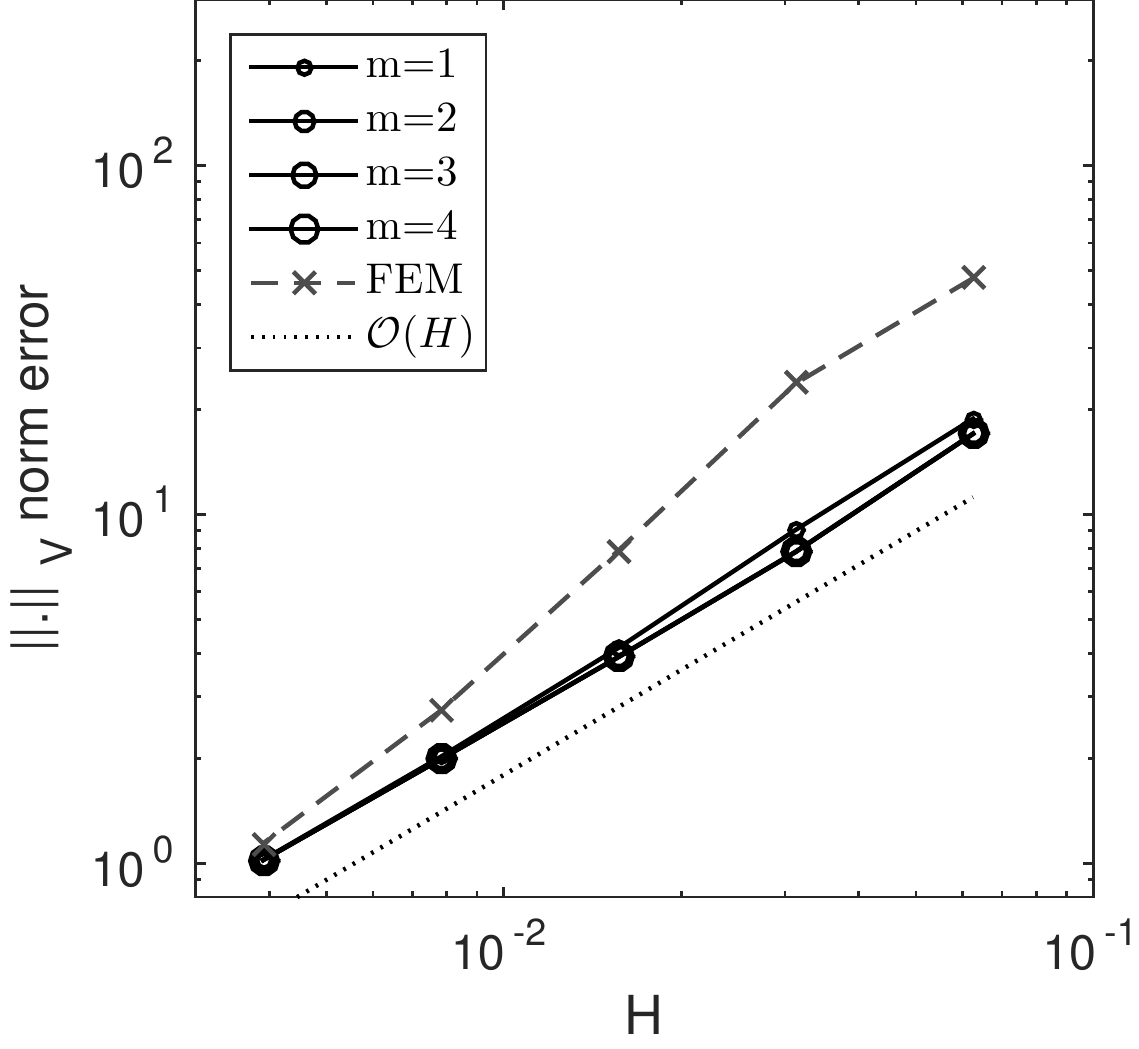}
\includegraphics[width=.49\textwidth]{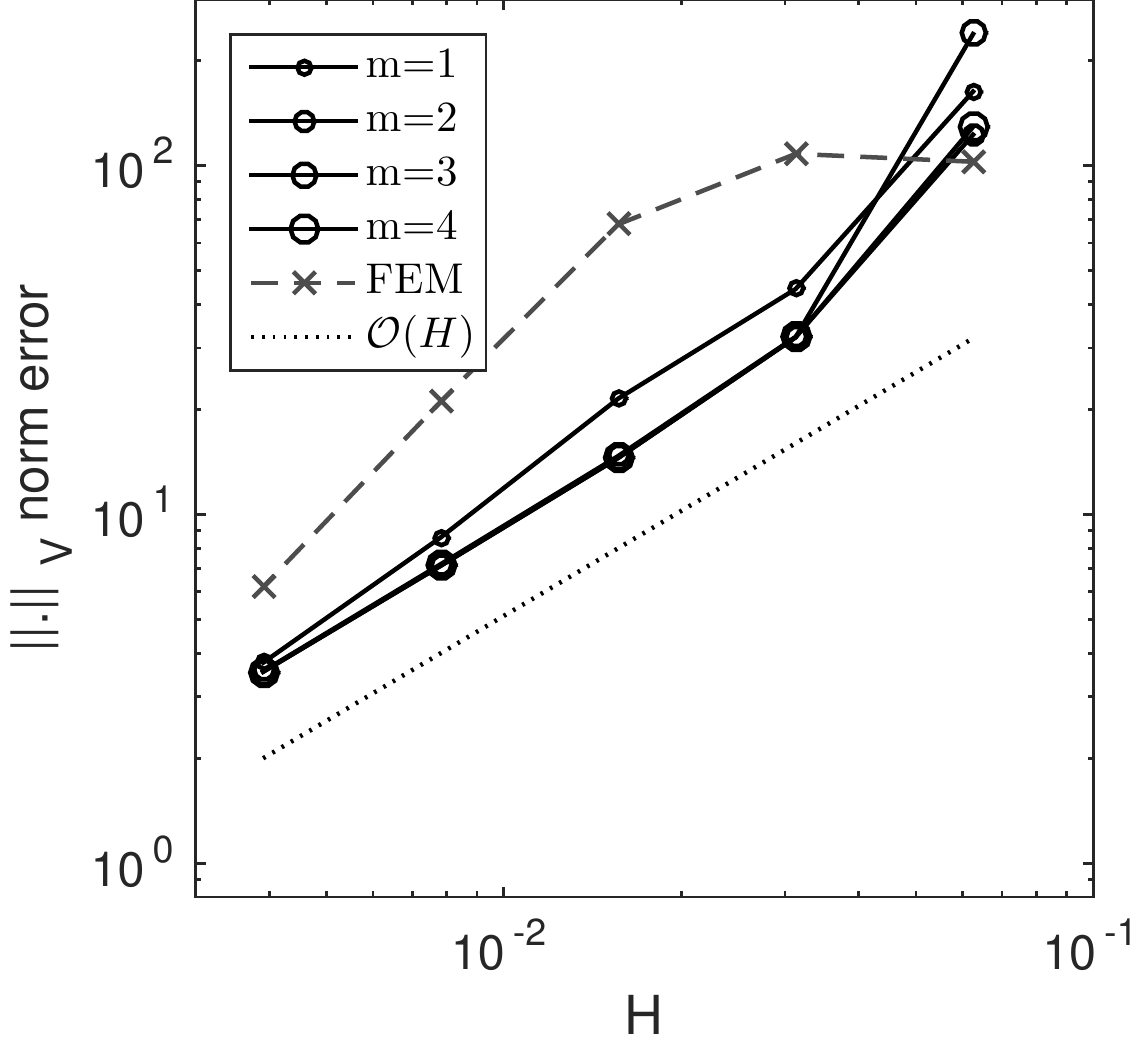}
\end{center}
\caption{Convergence history for the multiple scattering example
         from Subsection~\ref{ss:multScatterer} for 
         $\kappa=2^5$ (left) and $\kappa=2^6$ (right) $h=2^{-11}$.
         \label{f:scatteringconv}
         }
\end{figure}

\subsection{Plane Wave on the Cube Domain}
On the unit cube $\Omega=(0,1)^3$, we consider the pure Robin
problem with data given by the plane wave
$u(x)=\exp(-i\kappa x\cdot \frac{1}{\sqrt{38}}
 \left(\begin{smallmatrix}2\\3\\5\end{smallmatrix}\right))$.

We choose $\kappa=2^5$.
Figure~\ref{f:convergenceCube} compares the error of the 
msPGGEM $h=2^{-4}$ and $m\in\{1,2,3,4\}$ with the best-approximation
in the $\|\cdot\|_V$ norm and the error of the standard Galerkin
FEM.
Also in this example, the msPGFEM is pollution-free for the
oversampling parameter $m \geq 2$. The quasi-optimality constant appears
slightly larger than in 2D.
For the oversampling parameter $m=2$, the number of corrector problems
to be solved
for the finest mesh $\G_H$
is 343 out of $262\,144$
when no symmetry is exploited.

\begin{figure}
\begin{center}
\includegraphics[width=0.49\textwidth]{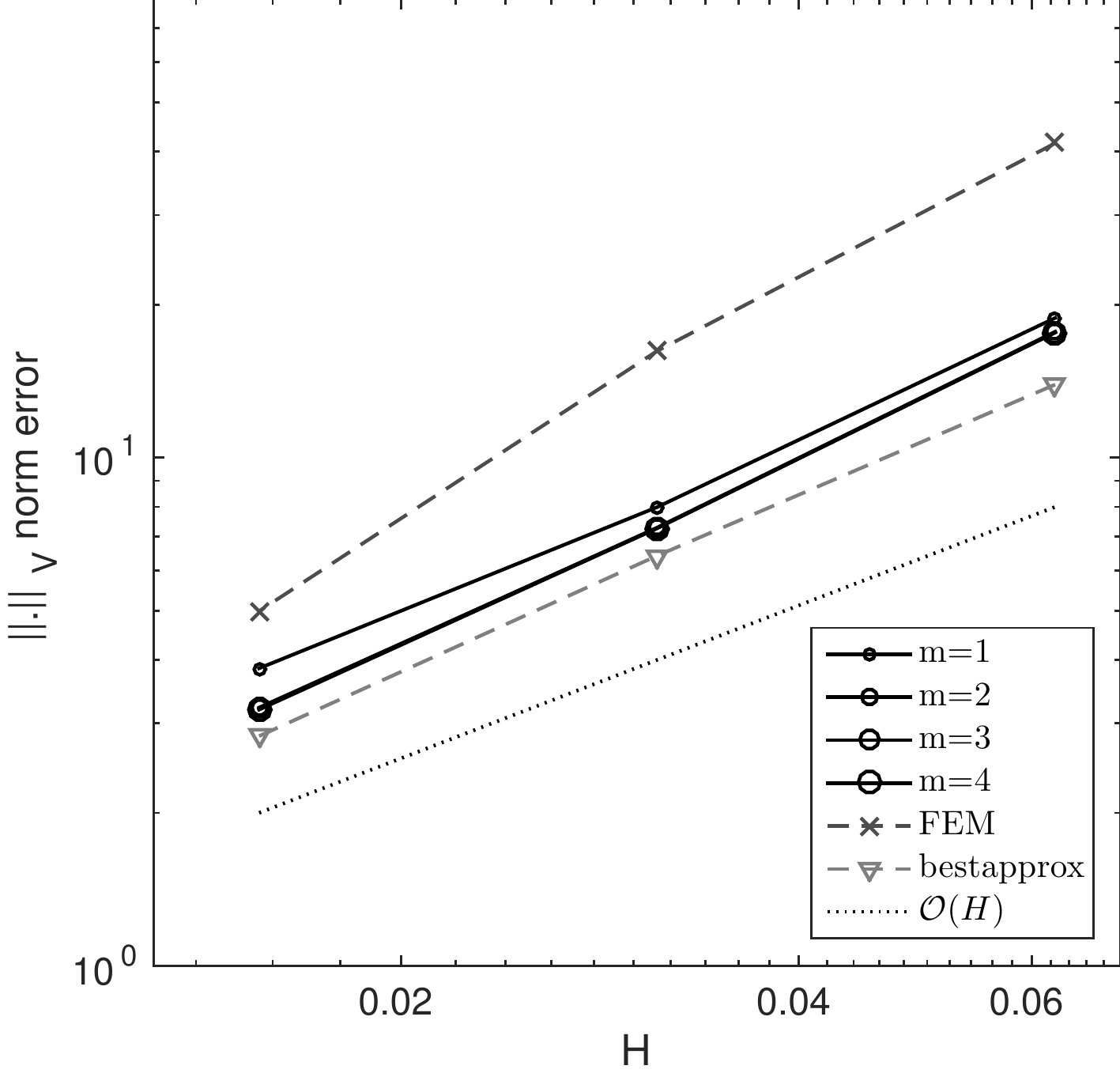}
\end{center}
\caption{Convergence history
         for the 3D plane wave example for $\kappa=2^5$ and $h=2^{-7}$.
         \label{f:convergenceCube}
         }
\end{figure}

\begin{appendix}

\section{Appendix: Proof of Theorem~\ref{t:CorrCloseness}}\label{a:proof}

For the sake of completeness we also present a proof of the exponential
decay result Theorem~\ref{t:decay} which is central for the method.
The idea of the proof is the same as in the previous proofs of the 
exponential decay
\cite{MP14,HP13,EGMP13,HMP14,Brown.Peterseim:2014}
in the context of diffusion problems.
The difference especially with respect to \cite{Peterseim2014} is
that here the quasi-interpolation is a projection.
This simplifies the proofs and leads to slightly better rates
in the exponential decay that have been experimentally observed 
in \cite{Peterseim2014}.

Let $I_h:C^0(\Omega)\to V_h$ denote the nodal $Q_1$ interpolation 
operator. Standard interpolation estimates and the inverse 
inequality prove for any $T\in \G_H$ and all
$q\in Q_2(T)$ the stability estimate
\begin{equation}\label{e:Ihestimate}
\|\nabla I_h q\|_{L^2(T)}
\leq \CIh \|\nabla q\|_{L^2(T)} .
\end{equation}
In the proofs we will frequently make use of cut-off functions.
We collect some properties in the following lemma.
\begin{lemma}
Let $\eta\in\cS^1(\G_H)$ be a function with values in the interval
$[0,1]$ satisfying the bound
\begin{equation}\label{e:etaestimate}
\|\nabla\eta\|_{L^\infty(\Omega)}\leq C_\eta H^{-1}
\end{equation}
and let $\cR:=\supp(\nabla\eta)$.
Given any subset $\mathcal K\subseteq\G_H$, 
any $\phi\in W_h$ satisfies for
$S=\cup\mathcal K\subseteq\overline\Omega$ 
 that
\begin{align}
\label{e:estA}
\| \phi\|_{L^2(S)}
& \lesssim H \|\nabla\phi\|_{L^2(\nei(S))}
\\
\label{e:estB}
\|(1-I_H)I_h(\eta\phi) \|_{L^2(S)}
& \lesssim H \|\nabla(\eta\phi)\|_{L^2(\nei(S))}
\\
\label{e:estC}
\|\nabla(\eta\phi)\|_{L^2(S)}
& \lesssim \|\nabla\phi\|_{L^2(S\cap\{\supp(\eta)\})} 
     + \|\nabla\phi\|_{L^2(\nei(S\cap\cR))} .
\end{align}
\end{lemma}

\begin{proof}
The property \eqref{e:IHapproxstab} readily implies \eqref{e:estA}.
Furthermore, \eqref{e:IHapproxstab} implies
\begin{equation*}
\|(1-I_H)I_h(\eta\phi) \|_{L^2(S)}
\leq H\CIH\sqrt{\Col} \|\nabla I_h(\eta\phi)\|_{L^2(\nei(S))}.
\end{equation*}
Estimate \eqref{e:Ihestimate} leads to
\begin{equation*}
\begin{aligned}
\|\nabla I_h(\eta\phi)\|_{L^2(\nei(S))}
\leq \CIh \|\nabla (\eta\phi)\|_{L^2(\nei(S))}.
\end{aligned}
\end{equation*}
This proves \eqref{e:estB}.
For the proof of \eqref{e:estC} the product rule and 
\eqref{e:etaestimate} imply
\begin{equation*}
\|\nabla(\eta\phi)\|_{L^2(S)} 
\leq
\|\nabla\phi\|_{L^2(S\cap\{\supp(\eta)\})}
+C_\eta H^{-1} \|\phi\|_{L^2(S\cap\cR)}.
\end{equation*}
The combination with \eqref{e:estA} concludes the proof.
\end{proof}

%%%%%%%%%%%%%%%%%%%%%%%%%%%%%%%%%%%%%%%%%%%%%%%%%%%%%%%%%%%%%%%%%%%%%%
\begin{theorem}[decay]\label{t:decay}
Under the resolution condition \eqref{e:resolution},
there exists $0<\beta<1$ such that,
for any $v_H\in V_H$ and all $T\in\G_H$ and $m\in\mathbb N$,
\begin{equation*}
\|\nabla \Cor_{T,\infty} v_H \|_{L^2(\Omega\setminus\nei^{m}(T))}
\leq 
C \beta^m \|\nabla v_H\|_{L^2(T)}.
\end{equation*}
\end{theorem}
\begin{proof}
We define the cut-off function $\eta\in \cS^1(\G_H)$ via
\begin{equation*}
\eta \equiv 0 \quad\text{in } \nei^{m-3}(T)
\qquad\text{and}\qquad
\eta \equiv 1 \quad\text{in } \Omega\setminus\nei^{m-2}(T).
\end{equation*}
Note that $\eta$ is thereby also uniquely defined on the set
$\cR:=\supp(\nabla\eta)$.
The shape-regularity implies that $\eta$ satisfies \eqref{e:etaestimate}.
Let $v_H\in V_H$ and denote
$\phi:=\Cor_{T,\infty} v_H\in W_h$.
Elementary estimates lead to
\begin{equation*}
\begin{aligned}
\|\nabla\phi\|_{\Omega\setminus{\nei^m(T)}}^2
\leq \lvert (\nabla\phi,\eta\nabla\phi)_{L^2(\Omega)}\rvert
&
\leq
 \lvert (\nabla\phi,\nabla(\eta\phi))_{L^2(\Omega)}\rvert
 + \lvert (\nabla\phi,\phi\nabla\eta)_{L^2(\Omega)}\rvert
\\
&
\leq
M_1+M_2+M_3+M_4
\end{aligned}
\end{equation*}
for
\begin{equation*}
\begin{aligned}
M_1&:=\lvert (\nabla\phi,\nabla((1-I_h)(\eta\phi)))_{L^2(\Omega)}\rvert
&M_2&:=\lvert (\nabla\phi,\nabla((1-I_H)I_h(\eta\phi)))_{L^2(\Omega)}\rvert
\\
M_3&:=\lvert (\nabla\phi,\nabla(I_H I_h(\eta\phi)))_{L^2(\Omega)}\rvert
\quad
&M_4&:=\lvert (\nabla\phi,\phi\nabla\eta)_{L^2(\Omega)}\rvert .
\end{aligned}
\end{equation*}
The property \eqref{e:Ihestimate} proves
\begin{equation*}
M_1
\leq 
\|\nabla\phi\|_{L^2(\cR)} \, \|\nabla(\eta\phi-I_h(\eta\phi))\|_{L^2(\cR)}
\lesssim 
\|\nabla\phi\|_{L^2(\cR)} \|\nabla(\eta\phi)\|_{L^2(\cR)} .
\end{equation*}
Hence, it follows with \eqref{e:estC} that
\begin{equation*}
M_1
\lesssim
\|\nabla\phi\|_{L^2(\cR)}  \|\nabla\phi\|_{L^2(\nei(\cR))} .
\end{equation*}
Since 
$w:=(1-I_H)I_h(\eta\phi) \in W_h$, the identity
\eqref{e:idealElementCorrProb} and the fact that the support
of $w$ lies outside
$T$ imply $a(w,\phi)=a_T(w,v_H)=0$ and therefore
\begin{equation*}
\begin{aligned}
M_2
=a(w,\phi)+\kappa^2(w,\phi)
=\kappa^2(w,\phi)
\leq
\kappa^2  \|w\|_{L^2(\nei(\cR))} \|\phi\|_{L^2(\nei(\cR))}.
\end{aligned}
\end{equation*}
The estimates \eqref{e:estA} and \eqref{e:estB} 
and the resolution condition $\kappa H \lesssim 1$ 
from \eqref{e:resolution}
imply
\begin{equation*}
M_2
\lesssim
    \|\nabla\phi\|_{L^2(\nei^2(\cR))} 
    \|\nabla(\eta\phi)\|_{L^2(\nei^2(\cR))}.
\end{equation*}
The application of \eqref{e:estC} yields
\begin{equation*}
M_2
\lesssim \|\nabla\phi\|_{L^2(\nei^2(\cR))} 
( \|\nabla\phi\|_{L^2(\nei^2(\cR))} + \|\nabla\phi\|_{L^2(\nei(\cR))}) 
\lesssim
\|\nabla\phi\|_{L^2(\nei^2(\cR))}^2 .
\end{equation*}
The function $I_H I_h (\eta\phi)$ vanishes outside $\nei(\cR)$.
Hence,
the stability and approximation properties \eqref{e:IHapproxstab}
and \eqref{e:Ihestimate} lead to
\begin{equation*}
\begin{aligned}
M_3 
&
\leq \|\nabla\phi\|_{L^2(\nei(\cR))} \|\nabla(I_H I_h(\eta\phi))\|_{L^2(\nei(\cR))}
\\
&
\lesssim
\|\nabla\phi\|_{L^2(\nei(\cR))}
      \|\nabla(\eta\phi)\|_{L^2(\nei^2(\cR))}.
\end{aligned}
\end{equation*}
With \eqref{e:estC} we obtain
\begin{equation*}
M_3
\lesssim
\|\nabla\phi\|_{L^2(\nei(\cR))}
(
\|\nabla\phi\|_{L^2(\nei^2(\cR))}
+ \|\nabla\phi\|_{L^2(\nei(\cR))}
)
\lesssim
\|\nabla\phi\|_{L^2(\nei^2(\cR))}^2 .
\end{equation*}
For the term $M_4$, the Lipschitz bound \eqref{e:etaestimate} and 
\eqref{e:estA} prove
\begin{equation*}
\begin{aligned}
M_4
\leq \|\nabla\phi\|_{L^2(\cR)} \, \|\phi\|_{L^2(\cR)} C_\eta H^{-1}
\lesssim
\|\nabla\phi\|_{L^2(\nei(\cR))}^2 .
\end{aligned}
\end{equation*}
Altogether, it follows for some constant
$\widetilde C$ that
\begin{equation*}
\|\nabla\phi\|_{L^2(\Omega\setminus{\nei^m(T)})}^2
\leq
\widetilde C
\|\nabla\phi\|_{L^2(\nei^2(\cR))}^2.
\end{equation*}
Recall that
$\nei^2(\cR) = \nei^m(T)\setminus\nei^{m-5}(T)$.
Since
\begin{equation*}
\|\nabla\phi \|_{L^2(\Omega\setminus\nei^{m}(T))}^2
+
\|\nabla\phi \|_{L^2(\nei^{m}(T)\setminus\nei^{m-5}(T))}^2
=
\|\nabla\phi \|_{L^2(\Omega\setminus\nei^{m-5}(T))}^2 ,
\end{equation*}
we obtain
\begin{equation*}
(1+\widetilde C^{-1})\|\nabla\phi \|_{L^2(\Omega\setminus\nei^{m}(T))}^2
\leq
\|\nabla\phi \|_{L^2(\Omega\setminus\nei^{m-5}(T))}^2.
\end{equation*}
The repeated application of this argument proves for
$\tilde\beta:=(1+\widetilde C^{-1})^{-1}<1$ that
\begin{equation*}
\|\nabla\phi \|_{L^2(\Omega\setminus\nei^{m}(T))}^2
\leq
\tilde\beta^{\lfloor  m/5\rfloor}
\|\nabla\phi \|_{L^2(\Omega)}^2
\lesssim
\tilde\beta^{\lfloor  m/5\rfloor} \|\nabla v_H \|_{L^2(T)}^2.
\end{equation*}
This is the assertion.
\end{proof}

We proceed with the proof of Theorem~\ref{t:CorrCloseness}.

\begin{proof}[Proof of Theorem~\ref{t:CorrCloseness}]
We define the cut-off function $\eta\in \cS^1(\G_H)$ via
\begin{equation*}
\eta \equiv 0 \quad\text{in } \Omega\setminus\nei^{m-1}(T)
\qquad\text{and}\qquad
\eta \equiv 1 \quad\text{in } \nei^{m-2}(T).
\end{equation*}
This function is thereby uniquely defined and satisfies the
bound \eqref{e:etaestimate}.
Since $(1-I_H)I_h(\eta\Cor_{T,\infty} v) \in W_h(\Omega_T)$,
we deduce with C\'ea's Lemma,
the identity $I_H\Cor_{T,\infty} v = 0$ and the approximation and stability
properties \eqref{e:IHapproxstab} and \eqref{e:Ihestimate} 
and the resolution condition \eqref{e:resolution} that
\begin{equation*}
\begin{aligned}
\|\nabla(\Cor_{T,\infty} v - \Cor_{T,m} v)\|_{L^2(\Omega)}^2
&
\lesssim
\|\Cor_{T,\infty} v - (1-I_H)I_h(\eta\Cor_{T,\infty} v)\|_V^2
\\
&
=
\|(1-I_H)I_h(\Cor_{T,\infty} v - \eta\Cor_{T,\infty} v)\|_{V,\Omega\setminus\{\eta=1\}}^2
\\
&
\lesssim
\|\nabla(1-\eta)\Cor_{T,\infty}v \|_{L^2(\nei(\Omega\setminus\{\eta=1\}))}^2
\\
&
\lesssim
\|\nabla\Cor_{T,\infty}v \|_{L^2(\nei(\Omega\setminus\{\eta=1\}))}^2.
\end{aligned}
\end{equation*}
Note that $\nei(\Omega\setminus\{\eta=1\}) = \Omega\setminus\nei^{m-3}(T)$.
This and Theorem~\ref{t:decay} prove \eqref{e:CorrCloseness1}.

Define $z:=(\Cor_\infty-\Cor_m)v$ and $z_T:=(\Cor_{T,\infty}-\Cor_{T,m})v$.
The ellipticity from Lemma~\ref{l:wellposedideal} proves
\begin{equation*}
\frac12 \|\nabla z\|_{L^2(\Omega)}^2
\leq
\biggl|
\sum_{T\in\G_H}a(z,z_T) 
\biggr|.
\end{equation*}
We define the cut-off function $\eta\in \cS^1(\G_H)$ via
\begin{equation*}
\eta \equiv 1 \quad\text{in } \Omega\setminus\nei^{m+2}(T)
\qquad\text{and}\qquad
\eta \equiv 0 \quad\text{in } \nei^{m+1}(T).
\end{equation*}
This function is thereby uniquely defined and satisfies the
bound \eqref{e:etaestimate}.
For any $T\in\G_H$ we have $(1-I_H)I_h(\eta z) \in W_h$ with support
outside $\Omega_T$.
Hence, we obtain with $z=I_h z$ that
\begin{equation*}
a(z,z_T)
=
a(I_h(z-\eta z),z_T)+a(I_H I_h(\eta z),z_T).
\end{equation*}
The function $z-I_h(\eta z)$ vanishes on $S:=\{\eta=1\}$. Hence,
the first term on the right-hand side satisfies
\begin{equation*}
\lvert a(I_h(z-\eta z),z_T)\rvert
\leq
C_a
\|I_h(z-\eta z)\|_{V,\Omega\setminus S} \|z_T\|_V.
\end{equation*}
The Friedrichs inequality with constant $C_{\mathrm{F}}$ proves
together with the stability \eqref{e:Ihestimate} and the
estimate \eqref{e:estC} applied to the cut-off function $(1-\eta)$
that
\begin{equation*}
\|I_h(z-\eta z)\|_{V,\Omega\setminus S}
\lesssim
\sqrt{1+(C_{\mathrm{F}}\kappa H)^2}\|\nabla z\|_{L^2(\Omega\setminus S)}
\lesssim
\|\nabla z\|_{L^2(\Omega\setminus S)} .
\end{equation*}
Furthermore, $I_H I_h(\eta z)$ vanishes on $\Omega\setminus\nei(\supp(1-\eta))$.
Hence, we infer from Friedrichs' inequality and the resolution condition
\eqref{e:resolution}, the stability properties
\eqref{e:IHapproxstab} and \eqref{e:Ihestimate} 
and the \eqref{e:estC} that
\begin{equation*}
\lvert a(z_T,I_H I_h(\eta z))\rvert
\lesssim
   \|\nabla z\|_{L^2(\nei^2(\supp(1-\eta)))} \|z_T\|_V .
\end{equation*}
The sum over all $T\in\G_H$ and the Cauchy inequality yield with 
the finite overlap of patches
\begin{equation*}
\begin{aligned}
\|\nabla z\|_{L^2(\Omega)}^2
&\lesssim
\sum_{T\in\G_H} \|\nabla z\|_{L^2(\nei^2(\supp(1-\eta)))} \|z_T\|_V 
\\
&
\lesssim
 \sqrt{\Colm}
 \|\nabla z\|_{L^2(\Omega)} \sqrt{\sum_{T\in\G_H}\|z_T\|_V^2}.
\end{aligned}
\end{equation*}
The combination with \eqref{e:CorrCloseness1} concludes the proof.
\end{proof}

\end{appendix}

{\footnotesize
\newcommand{\etalchar}[1]{$^{#1}$}

}

\end{document}